\documentclass[a4paper]{article}

\usepackage{amsthm}
\usepackage{amsmath,amsfonts,amssymb}
\usepackage{thmtools,thm-restate}
\usepackage{pstricks}
\usepackage{caption}

\newtheorem{lemma}{Lemma}[section]
\newtheorem{corollary}[lemma]{Corollary}
\newtheorem{theorem}[lemma]{Theorem}
\newtheorem{proposition}[lemma]{Proposition}

\theoremstyle{definition}
\newtheorem{definition}[lemma]{Definition}
\newtheorem{remark}[lemma]{Remark}

\newcommand{\Sphere}{\mathbb{S}^3}  
\newcommand{\sphere}{\mathbb{S}^2}  
\newcommand{\crcle}{\mathbb{S}^1}  
\newcommand{\mb}{\mathbb{M}}	
\newcommand{\kb}{\mathbb{K}}	
\newcommand{\pp}{\mathbb{P}}	
\newcommand{\ts}{\mathbb{T}}	
\newcommand{\disc}{\mathbb{D}} 
\DeclareMathOperator{\Int}{int} 
\DeclareMathOperator*{\homeo}{Homeo} 
\DeclareMathOperator{\Mod}{Mod} 
\DeclareMathOperator{\GL}{GL} 

\begin{document}

\title{The Birman exact sequence for $3$--manifolds}
\author{Jessica E. Banks}
\date{}
\maketitle
\begin{abstract}
We study the Birman exact sequence for compact $3$--manifolds, obtaining a complete picture of the relationship between the mapping class group of the manifold and the mapping class group of the submanifold obtained by deleting an interior point. This covers both orientable manifolds and non-orientable ones.
\end{abstract}


\section{Introduction}

The mapping class group $\Mod(M)$ of a manifold $M$ is the group of isotopy equivalence classes of self-homeomorphisms of $M$. 
In \cite{MR0243519}, Birman gave a relationship (now known as the `Birman exact sequence') between the mapping class group of a manifold $M$ and that of the same manifold with a finite number of points removed, focusing particularly on the case that $M$ is a surface. When only one puncture is added, the result is as follows.

\begin{theorem}[\cite{MR2850125} Theorem 4.6; see also \cite{MR0243519} Theorem 1 and Corollary 1.3, and \cite{MR0375281} Theorem 4.2]\label{birmanthm}
Let $M$ be a compact, orientable surface with $\chi(M)<0$. Choose \(p\in\Int(M)\) and set \(M_p=M\setminus\{p\}\). Then there is an exact sequence
\[
1\to \pi_1(M,p)\to \Mod(M_p)\to\Mod(M)\to 1.
\]
\end{theorem}

The aim of this paper is to prove the following analogue of this for compact $3$--manifolds.

\begin{restatable*}{theorem}{maintheorem}\label{maintheorem}
Let $M$ be a compact, connected $3$--manifold, not necessarily orientable. Choose \(p\in\Int(M)\) and set \(M_p=M\setminus\{p\}\). Then there exists a short exact sequence
\[
1\to K\to\pi_1(M,p)\to \Mod(M_p)\to\Mod(M)\to 1,
\]
where \(K\) is trivial unless $M$ is prime and is a Seifert bundle defined by an \(\crcle\)--action on \(M\) with fibres the orbits of the action.
In these cases \(K\) is generated by regular fibres of such bundle structures, and \(K\) is equal to 
$K=Z(\pi_1(M))\cap\pi_1^+(M)$, the intersection of the centre of \(\pi_1(M)\) with the subgroup of \(\pi_1(M)\) represented by orientation-preserving loops. 
When \(K\) is non-trivial it is cyclic (finite or infinite according to whether \(\pi_1(M)\) is) apart from the five special cases listed below; then it is \(\mathbb{Z}^2\) except in the last case where it is \(\mathbb{Z}^3\).

The exceptional cases are as follows:
\begin{itemize}
\item $M=\mathbb{T}^2\times I$;
\item $M=\mathbb{M}\times\crcle$ for \(\mathbb{M}\) the Mobius band;
\item \(M=\mathbb{K}\times\crcle\) for \(\mathbb{K}\) the Klein bottle;
\item \(M\) is the mapping torus of the involution of \(\mathbb{K}=\mathbb{M}\cup\mathbb{M}\) that interchanges the two copies of \(\mathbb{M}\).
\item $M=\mathbb{T}^3$. 
\end{itemize}
\end{restatable*}

The proof makes use of some major theorems about \(3\)--manifolds, including the Poincar\'e Conjecture, the Orthogonalisation Theorem, and the Seifert Fibre Space Theorem for both orientable and non-orientable manifolds.
Together these leave as the main outstanding issue the task of understanding the centre of the fundamental group of a Seifert bundle. We use the term `Seifert bundle' in the sense used by Whitten, \cite{MR1194999};
the definition is discussed in Section \ref{bundlessection}.
Work of Epstein shows that a \(3\)--manifold is a Seifert bundle (with this definition) if and only if it is foliated by circles (see \cite{zbMATH03894059} p.\! 429).

If the use of the Poincar\'e Conjecture and the Orthogonalisation Theorem are removed, the proof remains valid for all compact, connected \(3\)--manifolds whose prime connected summands with finite fundamental group are all either \(\Sphere\), \(\pp^2\times I\) or Seifert fibred.

We can consider Theorem \ref{maintheorem} in relation to the extension of Theorem \ref{birmanthm} to all compact connected surfaces. Suppose \(S\) is a compact, connected surface for which the group \(K_S\) given by the analogue of Theorem \ref{maintheorem} is non-trivial. Then for the \(3\)--manifold \(M=S\times\crcle\) the group \(K\) from Theorem \ref{maintheorem} will contain \(K_S\times\mathbb{Z}\). There are four such surfaces: an annulus, a Mobius band, a torus and a Klein bottle. Applying this construction yields four of the exceptional cases in the statement of Theorem \ref{maintheorem}. This final exception is a twisted product of \(\mathbb{K}\) and \(\crcle\).

\medskip

This paper is organised as follows. 
Section \ref{sequencessection}
discusses the existence of the exact sequence, and demonstrates that the centre of the fundamental group is a key object of interest for our purposes.
Section \ref{groupssection}
comprises technical results calculating group centres from presentations.
In Section \ref{bundlessection}
we discuss Seifert bundles and their fundamental groups. The centres of these groups are analysed in Section \ref{centresection}.
Section \ref{proofsection}
completes the proof of Theorem \ref{maintheorem}.

We will use the following notation: $\mathbb{S}^n$ is an $n$--sphere, $\mathbb{D}^n$ is an $n$--disc, $\mathbb{T}^n$ is an $n$--torus, $\mathbb{P}^n$ is real projective $n$--space, $\mb$ is a Mobius band, $\mathbb{K}$ is a Klein bottle, $\mathbb{A}$ is an annulus and $I$ is a closed interval.

\medskip

The idea for this paper was first suggested to me by a referee of a different paper. The development of this paper has also been aided by comments from and discussion with 
Steven Boyer,
Martin Bridson,
Allen Hatcher,
Andy Putman,
Saul Schleimer
and a referee of this paper.

\section{Exact sequences and centres}\label{sequencessection}

\begin{definition}\label{moddefn}
For a manifold $M'$, let $\homeo(M')$ denote the group of homeomorphisms from $M'$ to $M'$, and let $\homeo _0(M')$ be the subgroup of maps isotopic to the identity on $M'$.
Set $\Mod(M')=\homeo(M')/\homeo_0(M')$.
\end{definition}

\begin{remark}
There are a number of related groups that can be defined by placing restrictions on the homeomorphisms considered. A common such restriction, in the case of orientable manifolds, is to ask that all homeomorphisms are orientation-preserving. We will (mostly) not consider these other groups in this paper.
\end{remark}

Let $M$ be a compact, connected $3$--manifold. Fix a point $p\in\Int(M)$, and set \(M_p=M\setminus\{p\}\).
As in the proof of \cite{MR2850125} Theorem 4.6, there is a fibre bundle
\[
\homeo(M_p)\to\homeo(M)\to \Int(M).
\]
The long exact sequence of homotopy groups associated to this fibre bundle includes the sequence
\begin{align*}
\pi_1(\homeo(M))\to\pi_1(\Int(M))&\to\pi_0(\homeo(M_p))\\
&\to\pi_0(\homeo(M))\to\pi_0(\Int(M)).
\end{align*}
Since $M$ is connected, we therefore have the following result.

\begin{proposition}\label{sequenceprop}
There is an exact sequence of groups
\[
\pi_1(M,p)\to \Mod(M_p)\to\Mod(M)\to 1.
\]
\end{proposition}

Note that the second map is induced by the inclusion $M_p\to M$.
The first map is given by sending a loop $\rho$ in $M$ based at $p$ to the isotopy class of homeomorphisms that results from pushing the point $p$ once around $\rho$. We will denote this map by $\Phi_M$, and write $\phi_{\rho}=\Phi_M(\rho)$ for $\rho\in\pi_1(M)$. Our aim for the rest of this paper is to understand the kernel of $\Phi_M$.

Choose a loop $\rho$ based at $p$ such that $\rho\in\ker(\Phi_M)$. Then there exists an isotopy $H\colon M\times[0,1]\to M$ from the identity map on $M$ to itself such that $\rho(t)=H(p,t)$ for $t\in[0,1]$.
Choose a point $p'\in M$, and let $\rho'$ be the loop given by $\rho'(t)=H(p',t)$ for $t\in[0,1]$. Also choose a path $\sigma$ from $p$ to $p'$. Then $\rho$ is homotopic to $\sigma\cdot\rho'\cdot\sigma^{-1}$.
We can draw several conclusions from this.

\begin{lemma}
If, in Definition \ref{moddefn}, we chose to restrict ourselves to a class of homeomorphisms that all have a common fixed point other than $p$ (for example, taking only homeomorphisms that restrict to the identity on $\partial M$) then it would immediately follow that $\ker(\Phi_M)\cong 1$.
\end{lemma}
\begin{proof}
Let $p'$ be a common fixed point of the homeomorphisms, with $p'\neq p$. Choose $\rho$, $H$ and $\sigma$ as above, and define $\rho'$ accordingly.
Since $\rho'$ is the constant loop, we find that $\rho$ is null-homotopic.
\end{proof}

\begin{lemma}\label{centrelemma}
The kernel of $\Phi_M$ is contained in the centre $Z(\pi_1(M,p))$ of $\pi_1(M,p)$.
\end{lemma}
\begin{proof}
Choose $\rho\in\ker(\Phi_M)$, and $\sigma\in\pi_1(M,p)$.
View $\sigma$ as a path from $p$ to $p$. Then we see that $\rho$ is homotopic to $\sigma\cdot\rho\cdot\sigma^{-1}$, so $\rho\cdot\sigma=\sigma\cdot\rho$ in $\pi_1(M,p)$.
\end{proof}

\begin{remark}
We will write $Z_M$ for $Z(\pi_1(M,p))$.
In addition, we will denote by $\pi_1^+(M,p)$ the subgroup of $\pi_1(M,p)$ consisting of orientation-preserving loops, which is $\pi_1(M,p)$ if $M$ is orientable, and has index $2$ in $\pi_1(M,p)$ if $M$ is non-orientable.
\end{remark}

\begin{lemma}\label{orientationpreservinglemma}
$\ker(\Phi_M)\leq\pi_1^+(M,p)$.
\end{lemma}
\begin{proof}
Given $\rho\in\ker(\Phi_M)$, find an isotopy $H\colon M\times[0,1]\to M$ as above. The map $x\mapsto H(x,1)$ is the identity map on $M$, which is an orientation-preserving map. This implies that the loop $\rho$ is orientation-preserving.
\end{proof}

The results discussed so far are independent of the dimension of $M$. From now on we will be using that $M$ is a $3$--manifold.

\begin{proposition}\label{simplebdyprop}
If $M$ has a boundary component that is either $\sphere$ or $\pp^2$ then $\ker(\Phi_M)\cong 1$.
\end{proposition}
\begin{proof}
Let $\rho\in\ker(\Phi)$. Take $p'$ to be a point on a boundary component $S$ of $M$ that is either $\sphere$ or $\pp^2$.
Choose $H$ and $\sigma$ as above, and define $\rho'$ accordingly. Then $\rho$ is homotopic to $\sigma\cdot\rho'\cdot\sigma^{-1}$.
Note that $\rho'$ is an element of $\pi_1(S,p')$.
If $\rho'$ is null-homotopic in $S$ then $\rho$ is null-homotopic in $M$. 
This is necessarily the case if $S$ is $\sphere$.
If $S$ is $\pp^2$ then the other possibility is that $\rho'$ is the orientation-reversing element of $\pi_1(S,p')$.
As $S$ is a boundary component of $M$, it must be two-sided in $M$. Thus the orientation-reversing loop in $S$ is also orientation-reversing in $M$. Lemma \ref{orientationpreservinglemma} therefore tells us that this possibility cannot occur.
\end{proof}

\begin{corollary}\label{nonprimecor}
If $M$ is a non-trivial connected sum then either $Z_M\cong 1$ or $M$ has an $\sphere$ boundary component (and hence $\ker(\Phi_M)\cong 1$).
\end{corollary}
\begin{proof}
Suppose $M$ is a connected sum of $M_1$ and $M_2$, neither of which is \(\Sphere\). Then $\pi_1(M)=\pi_1(M_1)*\pi_1(M_2)$.
If $\pi_1(M_1)\ncong 1\ncong \pi_1(M_2)$ then $Z_M\cong 1$.
Otherwise assume, without loss of generality, that $\pi_1(M_1)\cong 1$.
Using the Poincar\'e Theorem, this implies that $M_1$ has at least one boundary component that is an $\sphere$, and therefore so does $M$.
Combining Lemma \ref{centrelemma} and Proposition \ref{simplebdyprop} shows that $\ker(\Phi_M)\cong 1$.
\end{proof}

\section{Groups}\label{groupssection}

Having concluded that the centre of a fundamental group is of interest, we here 
establish details of the centres of certain groups given by group presentations.
These will be applied in Section \ref{centresection}.

\begin{lemma}\label{cyclicgrouplemma}
If $r$ and $s$ are coprime then the group
\[
\langle a,b\mid ab=ba,a^r=b^s\rangle
\]
is isomorphic to \(\mathbb{Z}\), generated by $a^mb^n$ where $ms+nr=1$.
\end{lemma}
\begin{proof}
There is an isomorphism \(\theta\) to the group \(\langle c\mid \ \rangle\) given by \(\theta(a)=c^s\) and \(\theta(b)=c^r\). The inverse map satisfies \(\theta^{-1}(c)=a^m b^n\).
\end{proof}

\begin{lemma}\label{powertwolemma}
If $r$ and $s$ are coprime then the group
\[
\langle a,b\mid ab=ba,a^{2r}=b^{2s}\rangle
\]
is isomorphic to $\mathbb{Z}\times\mathbb{Z}_2$, generated by $a^{-r}b^{s}$ and $a^{m}b^{n}$ where $ms+nr=1$.
\end{lemma}
\begin{proof}
There is an isomorphism \(\theta\) to the group
\[
\langle c,d\mid c^2=1,cd=dc\rangle
\]
given by \(\theta(a)=c^{-n}d^s\) and \(\theta(b)=c^md^r\). The inverse map satisfies
\(\theta^{-1}(c)=a^{-r}b^s\) and \(\theta^{-1}(d)=a^m b^n\).
\end{proof}

\begin{lemma}\label{grouplemma2}
Let $G$ be the group with presentation
\[\langle a,b,t\mid at=ta,bt=tb,t^r ba=ab\rangle,\]
where $r\neq 0$.
Then $Z(G)=\langle t\rangle\cong\mathbb{Z}$.
\end{lemma}
\begin{proof}
Using the relations, we can rewrite any word in the form $a^x b^y t^z$ for some $x,y,z\in\mathbb{Z}$.
Moreover, 
\[a^{x}b^{y}t^{z}\cdot a^{x'}b^{y'}t^{z'}=a^{x+x'}b^{y+y'}t^{z+z'-r x' y}.\]

Define $\theta\colon G \to \GL(3,\mathbb{R})$ by
\[
\theta(a)=\left(\begin{matrix}1&0&0\\0&1&1\\0&0&1\end{matrix}\right),\quad
\theta(b)=\left(\begin{matrix}1&1&0\\0&1&0\\0&0&1\end{matrix}\right),\quad
\theta(t)=\left(\begin{matrix}1&0&\frac{1}{r}\\0&1&0\\0&0&1\end{matrix}\right).
\]
It is easily verified that this defines a homomorphism.
Then
\[\theta(a^xb^yt^z)=
\left(\begin{matrix}1&0&0\\0&1&x\\0&0&1\end{matrix}\right)
\left(\begin{matrix}1&y&0\\0&1&0\\0&0&1\end{matrix}\right)
\left(\begin{matrix}1&0&\frac{z}{r}\\0&1&0\\0&0&1\end{matrix}\right)
=\left(\begin{matrix}1&y&\frac{z}{r}\\0&1&x\\0&0&1\end{matrix}\right).
\]
From this we see that if $a^xb^yt^z=1$ in $G$ then $x=y=z=0$.
Hence \(a\), \(b\) and \(t\) all have infinite order.

It is clear that if $g\in\langle t\rangle$ then $g\in Z(G)$.
Conversely, let $g=a^{x}b^{y}t^{z}$ be an element of $Z(G)$.
Then, in $G$, 
\begin{align*}
1&= a g a^{-1} g^{-1}\\
&=a^{x+1}b^{y}t^{z}a^{-1}t^{-z}b^{-y}a^{-x}\\
&=a^{x+1}b^{y}a^{-1}b^{-y}a^{-x}\\
&=a^{x+1-1}b^{y-y}t^{-r(-1)(y)}a^{-x}\\
&=t^{r y}.
\end{align*}
Hence $y=0$.
Similarly
$1= b g b^{-1} g^{-1}=t^{-r x}$,
and so $x=0$.
Therefore $g=t^z\in\langle t\rangle$.
\end{proof}

\begin{lemma}\label{grouplemma3}
Let $G$ be the group with presentation
\[\langle a,b,t\mid at=t^{-1}a,bt=t^{-1}b,t^r ba=ab\rangle.\]
Then $Z(G)\cong\mathbb{Z}^2$, generated by $a^2$ and $b^2$ if $r$ is odd and by $a^2$ and $abt^{-r/2}$ if $r$ is even.
\end{lemma}
\begin{proof}
First note that $a^2$, $b^2$ and (if $r$ is even) $abt^{-r/2}$ are central in $G$, and $(abt^{-r/2})^2=ababt^{-r}=aabt^{-r}bt^{-r}=a^2b^2$.

Define $\phi\colon G\to\langle a',b',t'\mid a't'=(t')^{-1}a',b't'=(t')^{-1}b',(t')^{r+2} b'a'=a'b'\rangle$ by $\phi(a)=a't'$, $\phi(b)=b'$, $\phi(t)=t'$. Then $\phi$ is an isomorphism. Hence we may assume that $r>0$.

As in the proof of Lemma \ref{grouplemma2}, we can rewrite any word in the form $a^x b^y t^z$ for some $x,y,z\in\mathbb{Z}$.
Define $\theta\colon G \to \GL(4,\mathbb{R})$ by
\begin{align*}
\theta(a)=\left(\begin{matrix}-1&-1&0&1\\0&-1&0&0\\0&0&1&0\\0&0&0&1\end{matrix}\right)&,\quad
\theta(b)=\left(\begin{matrix}1&0&0&0\\0&1&0&0\\0&0&-1&-1\\0&0&0&-1\end{matrix}\right),\\
\theta(t)=&\left(\begin{matrix}1&0&0&\frac{2}{r}\\0&1&0&0\\0&0&1&0\\0&0&0&1\end{matrix}\right).
\end{align*}
This defines a homomorphism, and
\[\theta(a^xb^yt^z)
=\left(\begin{matrix}(-1)^x&x(-1)^x &0&f(x,y,z)\\0&(-1)^x&0&0\\0&0&(-1)^y&y(-1)^y \\0&0&0&(-1)^y\end{matrix}\right),
\]
where $f(x,y,z)=\frac{2z}{r}(-1)^x+\frac{1}{2}(-1)^y(1-(-1)^x)$.
From this we see that if $a^xb^yt^z=1$ in $G$ then $x=y=z=0$.
Moreover, 
\begin{align*}
\theta(a^xb^yt^z\cdot& a^{x'}b^{y'} t^{z'})
=\\
&\left(\begin{matrix}(-1)^{x+x'}&(x+x')(-1)^{x+x'} &0&g(x,y,z,x',y',z')
\\0&(-1)^{x+x'}&0&0\\0&0&(-1)^{y+y'}&(y+y')(-1)^{y+y'} \\0&0&0&(-1)^{y+y'}\end{matrix}\right),
\end{align*}
where \(g(x,y,z,x',y',z')=(-1)^xf(x',y',z')+(-1)^{y'}f(x,y,z)\).
Hence $a^xb^yt^z$ is central if and only if $g(x,y,z,x',y',z')=g(x',y',z',x,y,z)$ for all choices of $(x',y',z')$.
Varying $z'$ shows that this implies
$\frac{2}{r}((-1)^{x+x'}-(-1)^{y+x'})=0$, so $x+y$ must be even. In particular, \((-1)^x=(-1)^y\).
Setting $x'=0$ and $y'=1$ then gives the equation
$(-1)^x(\frac{8z}{r}+2)=2$.
If $x$ is even, this implies that $z=0$. On the other hand, if $x$ is odd this implies that $z=-\frac{r}{2}$, which is only possible if $r$ is even. From this we see that $Z(G)$ is generated by $a^2$, $b^2$ and (if $r$ is even) $abt^{-r/2}$.

We can see that any non-trivial element in the subgroup generated by \(a^2\) and \(b^2\) has infinite order, so this subgroup is a copy of \(\mathbb{Z}^2\).
The same is true for \(a^2\) and \(abt^{-r/2}\) when \(r\) is even, since
\[
\theta\left((a^2)^x(abt^{-r/2})^y
\right)=
\begin{pmatrix}
(-1)^y&(2x+y)(-1)^y&0&0\\
0&(-1)^y&0&0\\
0&0&(-1)^y&y(-1)^y \\
0&0&0&(-1)^y
\end{pmatrix}.\qedhere
\]
\end{proof}

\begin{lemma}\label{grouplemma7}
Let \(G\) be the group with presentation
\[\langle a,b,t\mid at=ta,bt=tb,t^r =a^2b^2\rangle.\]
Then \(Z(G)\cong\mathbb{Z}^2\), generated by \(t\) and \(a^2\), or alternatively by \(t\) and \(b^2\).
\end{lemma}
\begin{proof}
Note that $t$ and $a^2$ are central, 
and
\begin{align*}
G/\langle t,a^2\rangle=& \langle a,b\mid 1= a^2=b^2\rangle,
\end{align*}
which 
has no centre.
Thus $Z(G)$ is generated by $t$ and $a^2$, or alternatively by $t$ and $b^2$. 
There is a homomorphism \(\theta\colon G\to\GL(5,\mathbb{R})\) given by
\[
\theta(a)=
\begin{pmatrix}
-1&0&0&0&1\\
0&1&0&0&\frac{r}{2}\\
0&0&1&0&0\\
0&0&0&1&0\\
0&0&0&0&1
\end{pmatrix}, \quad
\theta(b)=
\begin{pmatrix}
1&0&0&0&1\\
0&1&0&0&0\\
0&0&-1&0&0\\
0&0&0&1&\frac{r}{2}\\
0&0&0&0&1
\end{pmatrix},
\]
\[
\theta(t)=
\begin{pmatrix}
1&0&0&0&0\\
0&1&0&0&1\\
0&0&1&0&0\\
0&0&0&1&1\\
0&0&0&0&1
\end{pmatrix}.
\]
Any element of \(Z(G)\) can be written as \(a^{2x}t^y\) for some \(x,y\in\mathbb{Z}\).
Since
\[
\theta\left( a^{2x} t^y \right)=
\begin{pmatrix}
1&0&0&0&0\\
0&1&0&0&x+y\\
0&0&1&0&0\\
0&0&0&1&y\\
0&0&0&0&1
\end{pmatrix},
\]
we see that \(Z(G)\cong\mathbb{Z}^2\).
\end{proof}

\begin{lemma}\label{grouplemma4}
If \(r\neq 0\) then the group \(G\) with presentation
\[\langle a,b,t\mid at=t^{-1}a, bt=t^{-1}b,t^r=a^2b^2\rangle\]
has no centre.
\end{lemma}
\begin{proof}
First note that the subgroup \(H\) generated by \(t\) and \(a^2\) is normal in \(G\), and
\[G/H=\langle a,b\mid 1=a^2=b^2\rangle,\]
which has no centre. Hence \(Z(G)\leq H\).

Any element of \(H\) can be expressed as \(a^{2x}t^z\) for some \(x,z\in\mathbb{Z}\).
There is a homomorphism \(\theta\colon G\to \GL(4,\mathbb{R})\) given by
\begin{align*}
\theta(a)=\left(\begin{matrix}-1&-1&-1&0\\0&-1&0&0\\0&0&1&-1\\0&0&0&1\end{matrix}\right)&,\quad
\theta(b)=\left(\begin{matrix}1&-1&0&0\\0&1&0&-1\\0&0&-1&-1\\0&0&0&-1\end{matrix}\right),\\
\theta(t)=&\left(\begin{matrix}1&0&0&\frac{2}{r}\\0&1&0&0\\0&0&1&0\\0&0&0&1\end{matrix}\right),
\end{align*}
and \(\theta(a^{2x}t^z)=I_4\) only if \(x=z=0\).
If \(a^{2x}t^z\) commutes with \(a\) then \(z=0\).
On the other hand, if \(\theta(a^{2x})\) commutes with \(\theta(b)\) then \(x=0\).
Thus \(Z(G)\cong 1\).
\end{proof}

\begin{lemma}\label{grouplemma5}
The group \(G\) with presentation
\[\langle a,b,t\mid at=t^{-1}a, bt=t^{-1}b,a^2=b^2=1\rangle\]
has no centre.
\end{lemma}
\begin{proof}
Every element of \(G\) can be expressed as
\(a^{\varepsilon}(ba)^x b^{\eta}t^z\) for some \(x\in\mathbb{Z}_{\geq 0}\), \(z\in\mathbb{Z}\) and \(\varepsilon,\eta\in\{0,1\}\).
There is a homomorphism \(\theta\colon G\to\GL(3,\mathbb{R})\) given by
\[
\theta(a)=\left(\begin{matrix}-1&0&0\\ 0&-1&0\\ 0&0&1\end{matrix}\right),\quad
\theta(b)=\left(\begin{matrix}-1&0&1\\ 0&-1&0\\ 0&0&1\end{matrix}\right),\quad
\theta(t)=\left(\begin{matrix}1&0&0\\ 0&1&1\\ 0&0&1\end{matrix}\right).
\]
Now
\(\theta(a^{\varepsilon}(ba)^x b^{\eta}t^z)=I_3\) only if \(\varepsilon=\eta=x=z=0\).
It follows that \(a^{\varepsilon}(ba)^x b^{\eta}t^z\) commutes with \(b\) only if \(\varepsilon=x=z=0\). Since \(\theta(b)\) does not commute with \(\theta(t)\), we see that \(Z(G)\cong 1\).
\end{proof}

\begin{lemma}\label{grouplemma6}
The group \(G\) with presentation
\[\langle a,b,t\mid a t=ta, bt=t^{-1} b, t^r=a^2, b^2=1\rangle\]
has no centre.
\end{lemma}
\begin{proof}
Any element of \(G\) can be expressed as \(a^{\varepsilon}(ba)^xb^{\eta}t^z\) for some \(x\in\mathbb{Z}_{\geq 0}\), \(z\in\mathbb{Z}\) and \(\varepsilon,\eta\in\{0,1\}\).
There is a homomorphism \(\theta\colon G\to\GL(3,\mathbb{R})\) given by
\[
\theta(a)=\left(\begin{matrix}-1&0&0\\0&1&r\\0&0&1\end{matrix}\right),\quad
\theta(b)=\left(\begin{matrix}-1&0&1\\0&-1&0\\0&0&1\end{matrix}\right),\quad
\theta(t)=\left(\begin{matrix}1&0&0\\0&1&2\\ 0&0&1\end{matrix}\right),
\]
and \(\theta(a^{\varepsilon}(ba)^xb^{\eta}t^z)=I_3\) only if \(x=z=\varepsilon=\eta=0\).
Now, if \(a^{\varepsilon}(ba)^xb^{\eta}t^z\) commutes with \(a\) then \(x=\eta=0\).
On the other hand, \(\theta(a^{\varepsilon}t^z)\) does not commute with \(\theta(b)\) unless \(\varepsilon=z=0\). Hence \(Z(G)\cong 1\).
\end{proof}

\begin{lemma}\label{equalsquareslemma}
Let \(G\) be the group with presentation
\[
\langle a,b,t\mid ta=a^{-1}t,tb=b^{-1}t,a^2=b^2\rangle.
\]
Then \(Z(G)\cong\mathbb{Z}\), generated by \(t^2\).
\end{lemma}
\begin{proof}
Note that \(t^2\in Z(G)\).
Further, the subgroup generated by \(t^2\) and \(ab^{-1}\) is normal in \(G\), and
\[
G/\langle t^2,ab^{-1}\rangle=\langle a,t\mid ta=a^{-1}t,t^2=1\rangle,
\]
which has no centre.
It remains only to check whether \((ab^{-1})^n\) is central for any \(n\in\mathbb{Z}\).

If \((ab^{-1})^n\) commutes with \(a\) then \((ab^{-1})^{2n}=1\).
On the other hand, there is a homomorphism \(\theta\colon G\to\GL(4,\mathbb{R})\) given by
\[
\theta(a)=\begin{pmatrix}
1&0&0&1\\
0&-1&0&0\\
0&0&1&0\\
0&0&0&1
\end{pmatrix}, \;
\theta(b)=
\begin{pmatrix}
1&0&0&1\\
0&-1&0&1\\
0&0&1&0\\
0&0&0&1
\end{pmatrix}, \;
\theta(t)=
\begin{pmatrix}
-1&0&0&0\\
0&1&0&0\\
0&0&1&1\\
0&0&0&1
\end{pmatrix}.
\]
Since 
\[
\theta(ab^{-1})=
\begin{pmatrix}
1&0&0&0\\
0&1&0&-1\\
0&0&1&0\\
0&0&0&1
\end{pmatrix},
\]
we see that \(ab^{-1}\) has infinite order.
Hence \(Z(G)\) is generated by \(t^2\), which also has infinite order.
\end{proof}

\begin{lemma}\label{squarereverselemma}
The group \(G\) with presentation
\[\langle a,b\mid a^2=1, ab^r=b^{-r}a\rangle\]
has no centre.
\end{lemma}
\begin{proof}
Any element of \(G\) can be expressed as \(a^{\varepsilon}b^y\) for some \(y\in\mathbb{Z}\) and \(\varepsilon\in\{0,1\}\).
There is a homomorphism \(\theta\colon G\to\GL(3,\mathbb{C})\) given by
\[
\theta(a)=\begin{pmatrix}
1&0&0\\
0&0&i\\
0&-i&0
\end{pmatrix}, \quad
\theta(b)=
\begin{pmatrix}
1&i&1\\
0&1&0\\
0&0&1
\end{pmatrix},
\]
and \(\theta(a^{\varepsilon}b^y)=I_3\) only if \(y=\varepsilon=0\).
Moreover, \(\theta(a^{\varepsilon}b^y)\) commutes with \(\theta(a)\) only if \(y=0\). Since \(\theta(a)\) does not commute with \(\theta(b)\), this implies that \(Z(G)\cong 1\).
\end{proof}

\section{Seifert bundles}\label{bundlessection}

In proving Theorem \ref{maintheorem}
the key class of manifolds to consider will be Seifert bundles. In this section we consider these manifolds and their fundamental groups. In Section \ref{centresection} we will study \(Z_M\) and \(Z_M\cap\pi_1^+(M)\) when \(M\) is a Seifert bundle.

\subsection{Seifert fibred spaces}

Seifert fibred spaces were first defined by Seifert (\cite{zbMATH02547002}, translated in \cite{zbMATH03736614}) in 1933, and they have since been extensively studied. See also \cite{MR0415619} or \cite{MR565450}, for example, for a more general discussion. These spaces play a key role in the JSJ decomposition and the Geometrization Theorem.

A \textit{fibred solid torus} is a solid torus expressed as a union of circles (called \textit{fibres}), given by taking a disc $\disc^2$, viewing $\disc^2\times [0,1]$ as the union of arcs $\{q\}\times [0,1]$, and then gluing $\disc^2\times\{0\}$ to $\disc^2\times\{1\}$ by a rotation of $2\pi r/s$ for some $r,s\in\mathbb{Z}$ (which we assume to be coprime with $s\geq 1$). 
A fibre in the boundary of the solid torus therefore forms a curve on the boundary torus with a slope that runs $s$ times around the longitude and $r$ times around the meridian.
The central fibre is said to be \textit{exceptional} if \(r\neq 0\) and $s>1$, and any fibre that is not exceptional is called \textit{regular}. That is, a fibre is regular if it has a neighbourhood that is a fibred solid torus with $s=1$, and is exceptional if no such neighbourhood exists.

A compact, connected $3$--manifold $M'$ is \textit{Seifert fibred} if it can be expressed as a union of circles such that each fibre has a neighbourhood of fibres that is a fibred solid torus. Note that a fibre in $\partial M'$ will lie on the boundary of such a solid torus neighbourhood.
Since $M'$ is compact, and exceptional fibres are isolated, there are only finitely many exceptional fibres in $M'$.

Collapsing each fibre to a point gives a $2$--orbifold that is a surface $S'$ with marked points denoting the exceptional fibres. We call this orbifold the \textit{base space}. The boundary of $S'$ is formed of the fibres in $\partial M'$.
Given a simple closed curve $\rho$ in $S'$ (away from the marked points), the union of the corresponding fibres in $M'$ is a closed embedded surface in $M'$. This surface is a union of annuli, and so is either a torus or a Klein bottle. We say that $\rho$ \textit{preserves the fibre} if the surface is a torus, and that it \textit{reverses the fibre} if the surface is a Klein bottle. 

If every loop in $S'$ preserves the fibre then we can give a coherent orientation to all the fibres, and from this we can construct an $\crcle$--action on $M'$. This action induces an isotopy from the identity map on $M'$ to itself that takes a point $p'$ in a regular fibre of $M'$ once around that fibre.

\begin{proposition}[\cite{MR0415619} Theorem 12.1 (see also \cite{MR0426001} Section 5.3, \cite{MR565450} VI.9 and VI.10)]\label{pi1prop}
Suppose that $M$ is Seifert fibred with base space $S$.
If $S$ is orientable with genus $x$, $y$ boundary components and $z$ marked points then $\pi_1(M)$ is given by the presentation
\begin{align*}
\pi_1(M)=& \langle a_1,b_1,\ldots,a_x,b_x,c_1,\ldots,c_z,d_1,\ldots,d_y,t\mid\\
& a_i t=t^{\delta_i} a_i,b_i t=t^{\varepsilon_i} b_i, c_j t=t c_j, d_k t=t^{\zeta_k} d_k,\\
& c_j^{r_j}=t^{s_j}, t^r=a_1b_1a_1^{-1}b_1^{-1} a_2\cdots a_x^{-1}b_x^{-1}c_1\cdots c_z d_1\cdots d_y\rangle
\end{align*}
for some $r, r_j, s_j \in\mathbb{Z}$ and $\delta_i, \varepsilon_i,\zeta_k\in\{\pm 1\}$.

If instead $S$ is non-orientable with $x$ cross-caps, $y$ boundary components and $z$ marked points then $\pi_1(M)$ is given by the presentation
\begin{align*}
\pi_1(M)=& \langle a_1,\ldots,a_x,c_1,\ldots,c_z,d_1,\ldots,d_y,t\mid\\
& a_i t=t^{\delta_i} a_i, c_j t=t c_j, d_k t=t^{\zeta_k} d_k,\\
& c_j^{r_j}=t^{s_j}, t^r=a_1^2 a_2^2\cdots a_x^2 c_1\cdots c_z d_1\cdots d_y\rangle
\end{align*}
for some $r, r_j, s_j \in\mathbb{Z}$ and $\delta_i,\zeta_k\in\{\pm 1\}$.

In each case, $t$ represents a regular fibre in $M$. 
If the loop $a_i$ in $S$ preserves the fibre then $\delta_i=1$. Conversely, if $a_i$ reverses the fibre then $\delta_i=-1$. 
Similarly, $\varepsilon_i=1$ if $b_i$ preserves the fibre, and $\varepsilon_i=-1$ if $b_i$ reverses the fibre.
Likewise, $\zeta_k=1$ if $d_k$ preserves the fibre, and $\zeta_k=-1$ if \(d_k\) reverses the fibre.
\end{proposition}

\begin{remark}
As no trivial loop in \(S\) can reverse the fibre, \(\prod \zeta_k=1\). 
Accordingly, if the manifold \(M\) has boundary (so \(y\geq 1\)), the presentation can be simplified by removing the generator \(d_y\) and the two relations in which it appears.
\end{remark}

\begin{remark}
If $\delta_i=\varepsilon_i=\zeta_k=1$ for each $i$ and $k$ then $t$ is central in $\pi_1(M)$. On the other hand, if either $\delta_i=-1$ or $\varepsilon_i=-1$ for some $i$ or $\zeta_k=-1$ for some $k$ then $t^n$ is not central in $\pi_1(M)$ for any $n\neq 0$ unless $t$ has finite order. Thus either $\langle t\rangle\leq Z_M$ or $\langle t\rangle\cap Z_M\cong 1$, if $t$ has infinite order.
\end{remark}

\begin{definition}[\cite{MR0426001} Section 5.3]
A Seifert fibred space $M'$ with base space $S'$ is said to be \textit{small} if one of the following holds:
\begin{itemize}
\item $S'$ is $\sphere$ ($M'$ is therefore orientable) with at most two marked points;
\item $S'$ is $\sphere$ 
with three marked points\footnote{The definition in \cite{MR0426001} only calls these spaces small if $\pi_1(M')$ is finite. The standard terminology has since changed.}; 
\item $S'$ is $\sphere$ with four marked points, and $M'$ is the result of doubling (along its one toral boundary component) the orientable $\crcle$--bundle over a Mobius band; 
\item $S'$ is $\ts^2$ with no marked points; 
\item $S'$ is $\pp^2$ with at most one marked point;
\item $S'$ is $\kb$ with no marked points.
\end{itemize}
If $M'$ is not small then it is said to be \textit{large}.
\end{definition}

\begin{proposition}[\cite{MR0426001} Section 5.3 Lemma 1]\label{largenormalprop}
Assume $M$ is Seifert fibred and closed. Let $t$ be the regular fibre. If $M$ is large then $\langle t\rangle$ is the unique maximal normal subgroup of $\pi_1(M)$ and $t$ has infinite order.
\end{proposition}

\subsection{Seifert bundles}\label{bundlessubsection}

For our purposes, we will need to consider a more general class of manifolds. A \textit{fibred solid Klein bottle} is analogous to a fibred solid torus, with $\disc^2\times\{0\}$ glued to $\disc^2\times\{1\}$ via a reflection rather than a rotation. The \textit{exceptional} fibres are those that pass through the fixed-point set of the reflection. Within the solid Klein bottle, the union of the exceptional fibres is a properly-embedded one-sided annulus.

We wish to generalise the definition of Seifert fibred to allow for fibred neighbourhoods that are fibred solid Klein bottles rather than fibred solid tori.
As previously mentioned, we will follow the convention in \cite{MR3178943} and \cite{MR1194999}, and call a compact $3$--manifold $M'$ expressed in this way a \textit{Seifert bundle}. 

The union of the exceptional fibres that have solid Klein bottle neighbourhoods is a properly-embedded surface $T'$ in $M'$. Each component of this surface is one-sided in $M'$, and is either an annulus, a torus, or a Klein bottle.

The base space $S'$ is now a surface with marked points and mirrors, where the mirrors denote the fibres with solid Klein bottle neighbourhoods. These mirrors lie on the boundary of the surface $S'$. A torus or Klein bottle of exceptional fibres will form a single boundary component of $S'$. Each properly-embedded annulus of exceptional fibres will form a (closed) sub-arc of a boundary component of $S'$. 
For a simple closed curve mirror, we say that the mirror \textit{preserves the fibre} if the corresponding component of $T'$ is a torus, and that it \textit{reverses the fibre} if the corresponding component of $T'$ is a Klein bottle.

Cutting the Seifert bundle $M'$ along the surface $T'$ will yield a Seifert fibred space $M''$. The base space $S''$ of $M''$ is given by simply forgetting the mirrors on $S'$.
By considering the effect of re-gluing the components of $T'$, we can find a presentation for the fundamental group of $M'$.

Let $M_1$ be a Seifert bundle, and let $T$ be a connected, properly embedded surface in $M_1$ formed of exceptional fibres. In addition, let $M_2$ be the Seifert bundle that results from cutting $M_1$ along $T$. We can calculate $\pi_1(M_1)$ in terms of a presentation $\langle G\mid R\rangle$ for $\pi_1(M_2)$. 

First assume that $T$ is an annulus. A regular neighbourhood of $T$ is a twisted $I$--bundle over an annulus, which is a solid Klein bottle. The boundary Klein bottle of this is divided into an annulus $T^*$ and two Mobius bands, where the Mobius bands are contained within $\partial M_1$ and $T^*$ is properly embedded in $M_1$.

Note that $T^*$ is separating in $M_1$, and $M_2$ is given by deleting the component of $M_1\setminus T^*$ that contains $T$. That is, we may form $M_1$ by attaching the fibred neighbourhood of $T$ to $M_2$ along $T^*$. The Seifert--van-Kampen Theorem gives that $\pi_1(M_1)$ is the pushout of $\langle G\mid R\rangle$ and $\langle e\mid\ \rangle$ along $\langle h\mid\ \rangle$ with the inclusion maps $h\mapsto e^2$ and $h\mapsto t$, where $t$ denotes a regular fibre in $M_2$.
This gives the presentation
\[
\langle G,e\mid R, e^2=t\rangle.
\]
That is, the presence of the annulus $T$ adds a new generator to $\pi_1(M_1)$ whose square is the regular fibre.

Next suppose instead that $T$ is a torus. A fibred regular neighbourhood of $T$ is a twisted $I$--bundle over a torus, and is bounded by a torus $T^*$ that is separating in $M_1$. In this case we find that $\pi_1(M_1)$ is the pushout of $\langle G\mid R\rangle$ and $\langle f,g\mid fgf^{-1}g^{-1}\rangle$ along $\langle h_1,h_2\mid h_1h_2h_1^{-1}h_2^{-1} \rangle$ with the inclusion maps $h_1\mapsto f, h_2\mapsto g^2$ and $h_1\mapsto d_n,h_2\mapsto t$, where $d_n$ denotes the loop that runs around the boundary component of the base space of $M_2$ corresponding to $T^*$. Note that, since $T^*$ is a torus, $\zeta_n=1$. 
This gives the presentation
\[
\langle G,f,g\mid R,fg=gf,d_n=f,g^2=t\rangle.
\]
The effect of re-gluing $T$, therefore, is to rename one generator coming from a boundary component of $M_2$, add a new generator whose square is the regular fibre, and make these two generators commute. It is worth noting that the re-gluing map depends only on the slope of the fibres on $T$, and not on any choice of generating set for $\pi_1(T)$.

Finally, suppose that $T$ is a Klein bottle. A fibred regular neighbourhood of $T$ is then a twisted $I$--bundle over a Klein bottle; 
the core curve of one cross-cap preserves the fibre and the other reverses it.
The boundary $T^*$ of this is a Klein bottle. This time $\pi_1(M_1)$ is the pushout of $\langle G\mid R\rangle$ and $\langle f,g\mid fgf^{-1}g\rangle$ along $\langle h_1,h_2\mid h_1h_2h_1^{-1}h_2 \rangle$ with the inclusion maps $h_1\mapsto f, h_2\mapsto g^2$ and $h_1\mapsto d_n,h_2\mapsto t$. In this case, as $T^*$ is a Klein bottle, $\zeta_n=-1$.
This gives the presentation
\[
\langle G,f,g\mid R,fg=g^{-1}f,d_n=f,g^2=t\rangle.
\]
The effect of re-gluing $T$, therefore, is to rename one generator coming from a boundary component of $M_2$, add a new generator whose square is the regular fibre, and make the first of these generators conjugate the second to its inverse. Again, the definition of the re-gluing map depends only on the slope of the fibres on $T$.

Beginning with Proposition \ref{pi1prop} and applying these calculations inductively gives the following result.

\begin{proposition}\label{sbpi1prop}
Suppose $M$ is a Seifert bundle with base space $S$.
If $S$ is orientable, and has genus $x$, $y$ boundary components that are not entirely mirrors, $z$ marked points, $v$ arc mirrors and $w$ simple closed curve mirrors then
\begin{align*}
\pi_1(M)=& \langle a_1,b_1,\ldots,a_x,b_x,c_1,\ldots,c_z,d_1,\ldots,d_y,e_1,\ldots e_v, f_1,g_1,\ldots,f_w,g_w, t\mid\\
& a_i t=t^{\delta_i} a_i,b_i t=t^{\varepsilon_i} b_i, c_j t=t c_j, d_k t=t^{\zeta_k} d_k,\\
& c_j^{r_j}=t^{s_j}, e_l^2=t, g_m^2=t, f_m g_m=g_m^{\eta_m} f_m,\\
& t^r=a_1b_1a_1^{-1}b_1^{-1} a_2\cdots a_x^{-1}b_x^{-1}c_1\cdots c_z d_1\cdots d_y f_1\cdots f_w\rangle.
\end{align*}
If $S$ is non-orientable, and has $x$ cross-caps, $y$ boundary components that are not entirely mirrors, $z$ marked points, $v$ arc mirrors, and $w$ simple closed curve mirrors then
\begin{align*}
\pi_1(M)=& \langle a_1,\ldots,a_x,c_1,\ldots,c_z,d_1,\ldots,d_y,e_1,\ldots e_v, f_1,g_1\ldots,f_w,g_w,t\mid\\
& a_i t=t^{\delta_i} a_i, c_j t=t c_j, d_k t=t^{\zeta_k} d_k,\\
& c_j^{r_j}=t^{s_j},e_l^2=t,g_m^2=t, f_m g_m=g_m^{\eta_m} f_m,\\ 
& t^r=a_1^2 a_2^2\cdots a_x^2 c_1\cdots c_z d_1\cdots d_y f_1\cdots f_w\rangle.
\end{align*}
\end{proposition}

\begin{remark}\label{orientremark}
Here each $\delta_i$, $\varepsilon_i$, $\zeta_k$ or $\eta_m$ is $1$ if the corresponding curve in $S$ preserves the fibre and is $-1$ if the curve reverses the fibre. Moreover, because a trivial loop in $S$ can never reverse the fibre, $\prod \zeta_k \prod \eta_m =1$.

The relations $g_m^2=t$ and $f^{}_mg^{}_m=g_m^{\eta_m}f^{}_m$ ensure that $g^{}_m t g_m^{-1}=t$ and $f^{}_m t f_m^{-1}=t^{\eta_m}$ for each $m$.
Similarly, $e^{}_l t e_l^{-1}=t$ for each $l$.
We therefore see that $t$ generates a normal subgroup of $\pi_1(M)$, as for Seifert fibred spaces.
Moreover, $t\in Z_M$ if every $\delta_i$, $\varepsilon_i$, $\zeta_k$ and $\eta_m$ is $1$ (that is, if every loop in the base space preserves the fibre). On the other hand, if some $\delta_i$, $\varepsilon_i$, $\zeta_k$ or $\eta_m$ is $-1$ then either $\langle t\rangle\cap Z_M\cong 1$ or $t$ has finite order.

As for Seifert fibred manifolds, if every $\delta_i$, $\varepsilon_i$, $\zeta_k$ and $\eta_m$ is $1$ then we can give a coherent orientation to all the fibres, and from this we can construct an $\crcle$--action on $M$. Moreover, this action induces an isotopy from the identity map on $M$ to itself that takes a point of a regular fibre once around that fibre.
\end{remark}

\begin{remark}
As with Proposition \ref{pi1prop}, if \(M\) is not closed we may simplify the presentation by removing \(d_y\) and the two relations in which it appears.
\end{remark}

In order to understand some of the results in Section \ref{centresection}, it is helpful to study more precisely the gluing maps we have considered above. 

First suppose that the component $T$ along which we cut $M_1$ is a Klein bottle. The corresponding boundary component, $T^*$, of $M_2$ is also a Klein bottle, and is fibred by the Seifert bundle structure on $M_2$. The gluing map is given by taking each fibre and identifying points that are diametrically opposed. 

Now, up to isotopy, there are only four non-trivial simple closed curves on a Klein bottle. One is separating, dividing \(\kb\) into two Mobius bands. Two more are disjoint from this one and from each other, and are the core curves of the two Mobius bands. These are non-separating and orientation-reversing in \(\kb\). The final curve is non-separating and orientation-preserving in \(\kb\). It intersects the first curve twice, and the other two each once.

The construction of a fibred solid Klein bottle implies that the fibres on $T^*$ are non-separating and orientation-preserving in $T^*$. Thus the fibres must be in the isotopy class of the final curve we identified on the Klein bottle. The gluing map identifies the other two non-separating curves. In other words, $M_1$ can be defined from $M_2$ by simply specifying which boundary component of the base space is to become a mirror; no further details of the gluing map are required.

Suppose instead that $T$ is a torus. Then $T^*$ is a torus, again fibred by the fibres of the Seifert bundle structure on $M_2$. This time we do not have a-priori knowledge of the slope of the fibre. Never-the-less, given a fixed choice of a meridian-longitude pair $(\mu,\lambda)$ on $T^*$, there are not many choices for the gluing map. This can be seen as follows. If the fibre has slope $(r,s)$ (with $r$ and $s$ coprime), the gluing map can be viewed as translation by the vector $\frac{1}{2}(r,s)$ in the universal cover $\mathbb{R}^2$ of $T^*$. Once this is considered in the torus, rather than the universal cover, we see that there are only three options for this translation, given by the parity of each of $r$ and $s$. To re-iterate, although there are infinitely many possibilities for the slope $(r,s)$, for the purposes of defining the correct gluing map it is sufficient to assume that $(r,s)\in\{(0,1),(1,0),(1,1)\}$. We can further restrict these options if we have additional information connecting the slope of the fibre to the chosen pair $(\mu,\lambda)$. For example, if the fibre slope intersects $\mu$ once then we can reduce to considering $(r,s)\in\{(0,1),(1,1)\}$.

Let us now compare this with a related situation. Suppose that, instead of being formed of exceptional fibres, the torus $T$ along which we have cut is formed of the regular fibres over a simple closed curve $\rho$ in the base space. By definition, $\rho$ therefore preserves the fibre. Assume further that $\rho$ is orientation-reversing in the base space. With these assumptions, cutting $M_1$ along $T$ again results in a single new torus boundary component $T^*$ in $M_2$. The gluing map needed to return from $M_2$ to $M_1$ has the same form as those we have just considered. That is, the map can be viewed as a translation by a vector $\frac{1}{2}(r,s)$ in the universal cover of $T^*$. The difference here though is that, whereas previously the translation vector was given by the slope of the fibre, this time the translation is along a different slope on $T^*$ that intersects the fibre on $T^*$ once.

From this we see that, at least a-priori, it is possible that a torus in a manifold $M$ may consist of regular fibres in one Seifert bundle structure on $M$ and consist of exceptional fibres in an alternative Seifert bundle structure on $M$. We will see examples in Section \ref{centresection} where this does in fact occur.

A natural question is whether the same can be said of a Klein bottle within $M$. However, cutting a Seifert bundle along a Klein bottle of regular fibres gives a new boundary component that is a torus, as opposed to the Klein bottle that results from cutting a Klein bottle of exceptional fibres.

\section{The centre of $\pi_1(M)$ for a Seifert bundle}\label{centresection}

We now recall some results about finite-order elements of $\pi_1(M)$ that we will need in the following sections.

\begin{theorem}[\cite{MR0415619} Corollary 9.9]\label{finiteorderthm}
If $M$ is prime and $\pi_1(M)$ is infinite then any (non-trivial) element of finite order in $\pi_1(M)$ has order $2$. If $M$ contains no two-sided $\pp^2$ then $\pi_1(M)$ contains no non-trivial elements of finite order.
\end{theorem}

The following result follows from the proof of the Geometrization Conjecture, using earlier work on elliptic manifolds. The non-orientable case relies on work of Epstein, \cite{zbMATH03180034}. 
See also \cite{zbMATH06472948} Theorem 1.12 and the following discussion.

\begin{theorem}[\cite{MR3178943} Theorem 8]\label{finitepi1thm}
Suppose $M$ has no spherical boundary components, and that $\pi_1(M)$ is finite. If $M$ is orientable then it is Seifert fibred and closed. If $M$ is non-orientable then it is $\pp^2\times[0,1]$.
\end{theorem}

Note that $\pp^2\times[0,1]$ is not a Seifert bundle, as the boundary of a Seifert bundle is formed of tori and Klein bottles.

Work of Tollefson (\cite{zbMATH03461149}) gives the following result.

\begin{theorem}[\cite{MR3178943} Theorem 2]\label{rp2thm}
Suppose that $M$ is a Seifert bundle. If $M$ is reducible then $M$ is either $\pp^3\#\pp^3$ or an $\sphere$--bundle over $\crcle$. If $M$ is irreducible but contains a two-sided $\pp^2$ then $M$ is $\pp^2\times\crcle$.
\end{theorem}

\subsection{Seifert bundles with boundary}

\begin{lemma}\label{sbwithbdylemma}
Suppose $M$ is a Seifert bundle with base space $S$.
Let $t\in\pi_1(M)$ be given by a regular fibre in $M$.
Assume that $S$ has at least one boundary component.
Moreover, assume that $S$ is not any of the following:
\begin{itemize}
\item an annulus with no marked points and no mirrors;
\item a disc with one marked point and no mirrors; 
\item a disc with one arc mirror and no marked points;
\item an annulus with one boundary component a simple closed curve mirror and no marked points;
\item a Mobius band with no marked points and no mirrors.
\end{itemize}
If $t\in Z_M$ then $Z_M=\langle t\rangle$, whereas if $t\notin Z_M$ then $Z_M\cong 1$. Moreover, $t\in Z_M$ if and only if the fibres can be coherently oriented.
\end{lemma}
\begin{proof}
Consider the presentation for $\pi_1(M)$ given by Proposition \ref{sbpi1prop}.
Recall that $t$ generates a normal subgroup of $\pi_1(M)$. 
We will show that $Z(\pi_1(M)/\langle t\rangle)\cong 1$, which implies that $Z_M\leq\langle t\rangle$. 

First assume that $S$ is orientable. 
Then
\begin{align*}
\pi_1(M)/\langle t\rangle=& \langle a_1,b_1,\ldots,a_x,b_x,c_1,\ldots,c_z,d_1,\ldots,d_y,e_1,\ldots e_v, f_1,g_1,\ldots,f_w,g_w\mid\\
& c_j^{r_j}=1, e_l^2=1, g_m^2=1, f_m g_m=g_m^{\eta_m} f_m,\\
& 1=a_1b_1a_1^{-1}b_1^{-1} a_2\cdots a_x^{-1}b_x^{-1}c_1\cdots c_z d_1\cdots d_y f_1\cdots f_w\rangle
\end{align*}
for some $r_j \in\mathbb{Z}$ with $|r_j|\geq 2$.
Since $y\geq 1$, in fact 
\begin{align*}
\pi_1(M)/\langle t\rangle=& \langle a_1,b_1,\ldots,a_x,b_x,c_1,\ldots,c_z,d_1,\ldots,d_{y-1},\\
& e_1,\ldots e_v, f_1,g_1,\ldots,f_w,g_w\mid\\
& c_j^{r_j}=1, e_l^2=1, g_m^2=1, f_m g_m=g_m^{\eta_m} f_m\rangle.
\end{align*}
This is a free product of $2x+y-1$ copies of $\mathbb{Z}$, $z+v$ finite cyclic groups, and $w$ copies of either $\mathbb{Z}\times\mathbb{Z}_2$ or $\mathbb{Z}\rtimes\mathbb{Z}_2$.
It follows that the centre of $\pi_1(M)/\langle t\rangle$ is trivial unless $2x+y-1+z+v+w\leq 1$. As $y\geq 1$, the possible exceptions are given by $(x,y,z,v,w)\in\{(0,1,0,0,0),(0,2,0,0,0),(0,1,1,0,0),(0,1,0,1,0),(0,1,0,0,1)\}$.
If $(x,y,z,v,w)=(0,1,0,0,0)$ then $M$ is a solid torus and $\pi_1(M)$ is generated by $t$.
The other cases are excluded by hypothesis.

In addition, there is a homomorphism \(\theta\colon\pi_1(M)\to\mathbb{R}\rtimes\mathbb{Z}_2\) given by
\begin{align*}
\theta(a_i)=(0,(1-\delta_i)/2),\quad &\theta(b_i)=(0,(1-\varepsilon_i)/2),\quad\theta(c_j)=(s_j/r_j,0),\\
\theta(e_l)=(1/2,0),\quad &\theta(f_m)=(0,(1-\eta_m)/2),\quad\theta(g_m)=(1/2,0),\\
\theta(t)=(1,0),\quad&\theta(d_k)=(0,(1-\zeta_k)/2) \text{ for }k<y,\\
\theta(d_y)=\theta(d_{y-1}^{-1}\cdots d_1^{-1}& c_z^{-1}\cdots c_1^{-1}b_xa_x\cdots a_2^{-1}b_1a_1b_1^{-1}a_1^{-1}t^rf_w^{-1}\cdots f_1^{-1}).
\end{align*} 
Note that \(\theta(d_y t d_y^{-1}t^{-\zeta_y})=(0,0)\) because the condition that \(\prod \zeta_k\prod\eta_m=1\) ensures that \(\theta(d_y)=(\alpha,(1-\zeta_y)/2)\) for some \(\alpha\in\mathbb{R}\).
From this we see that \(t\) has infinite order in \(\pi_1(M)\).

Now assume instead that $S$ is non-orientable. 
Then
\begin{align*}
\pi_1(M)/\langle t\rangle=& \langle a_1,\ldots,a_x,c_1,\ldots,c_z,d_1,\ldots,d_y,e_1,\ldots e_v, f_1,g_1\ldots,f_w,g_w\mid\\
& c_j^{r_j}=1,e_l^2=1,g_m^2=1, f_m g_m=g_m^{\eta_m} f_m,\\ 
& 1=a_1^2 a_2^2\cdots a_x^2 c_1\cdots c_z d_1\cdots d_y f_1\cdots f_w\rangle.
\end{align*}
Again, since $y\geq 1$, 
\begin{align*}
\pi_1(M)/\langle t\rangle=& \langle a_1,\ldots,a_x,c_1,\ldots,c_z,d_1,\ldots,d_{y-1},e_1,\ldots e_v, f_1,g_1\ldots,f_w,g_w\mid\\
& c_j^{r_j}=1,e_l^2=1,g_m^2=1, f_m g_m=g_m^{\eta_m} f_m\rangle.
\end{align*}
This is a free product of $x+y-1$ copies of $\mathbb{Z}$, $z+v$ finite cyclic groups, and $w$ copies of either $\mathbb{Z}\times\mathbb{Z}_2$ or $\mathbb{Z}\rtimes\mathbb{Z}_2$, and so has trivial centre unless $x+y-1+z+v+w\leq 1$.
Note that $x\geq 1$, so the only exceptional case is $(x,y,z,v,w)=(1,1,0,0,0)$, which is excluded by hypothesis.

As before we can see that \(t\) has infinite order using the homomorphism \(\theta\colon\pi_1(M)\to\mathbb{R}\rtimes\mathbb{Z}_2\) given by
\begin{align*}
\theta(a_i)=(0,(1-\delta_i)/2),\quad&\theta(c_j)=(s_j/r_j,0),\\
\theta(e_l)=(1/2,0),\quad &\theta(f_m)=(0,(1-\eta_m)/2),\quad\theta(g_m)=(1/2,0),\\
\theta(t)=(1,0),\quad&\theta(d_k)=(0,(1-\zeta_k)/2) \text{ for }k<y,\\
\theta(d_y)=\theta(d_{y-1}^{-1}\cdots d_1^{-1}& c_z^{-1}\cdots c_1^{-1}a_x^{-2}\cdots a_1^{-2}t^rf_w^{-1}\cdots f_1^{-1}).
\end{align*} 

By Remark \ref{orientremark}, in both cases we find that either \(Z_M=\langle t\rangle\) or \(Z_M\cong 1\).
\end{proof}

\begin{remark}
We could alternatively verify that \(t\) has infinite order using Theorem \ref{finiteorderthm}. Since $M$ has a boundary component that is either a torus or a Klein bottle, Theorem \ref{finitepi1thm} shows that $\pi_1(M)$ is infinite, and Theorem \ref{rp2thm} shows that $M$ is prime and contains no two-sided $\pp^2$.
\end{remark}

\begin{lemma}\label{isZhasbdylemma}
Suppose that $M$ is a Seifert bundle with base space $S$. Suppose that one of the following holds:
\begin{itemize}
\item $S$ is a disc with one marked point and no mirrors;
\item $S$ is a disc with no marked points and one arc mirror;
\item $S$ is a Mobius band with no marked points or mirrors, and $M$ is orientable;
\item $S$ is an annulus with no marked points or mirrors, and $M$ is non-orientable;
\item $S$ is an annulus with no marked points and one boundary component a simple closed curve mirror, and $M$ has boundary a Klein bottle.
\end{itemize}
Then $Z_M\cap \pi_1^+(M)\cong\mathbb{Z}$, generated by a regular fibre of some Seifert bundle structure in which the fibres can be coherently oriented.
\end{lemma}
\begin{proof}
First suppose that $S$ is a disc with one marked point.
Then $M$ is a solid torus (and hence also Seifert fibred over a disc with no marked points) and $Z_M=\pi_1(M)\cong \mathbb{Z}$.

If $S$ is a disc with no marked points and one arc mirror,
then $M$ is a (fibred) solid Klein bottle, 
$Z_M=\pi_1(M)= \langle e_1, t\mid e_1^2=t\rangle$ and $Z_M\cap\pi_1^+(M)=\pi_1^+(M)\cong\mathbb{Z}$, generated by $t$.

Now suppose that $S$ is a Mobius band with no marked points or mirrors, and $M$ is orientable (equivalently, the core curve of the Mobius band reverses the fibre). It is noted in \cite{MR0426001} Section 5.4 that in this case $M$ is also Seifert fibred over a disc with two marked points. We may therefore apply Lemma \ref{sbwithbdylemma}.

Suppose next that $S$ is an annulus with no marked points or mirrors, and $M$ is non-orientable. Then $M$ is $\kb\times I$ with, 
in the notation of Proposition \ref{pi1prop},
\begin{align*}
\pi_1(M)
=& \langle d_1,t\mid d_1 t=t^{-1} d_1\rangle.
\end{align*}
Since \(d_1^2\) is central and
\begin{align*}
\pi_1(M)/\langle d_1^2\rangle=& \langle d_1,t\mid d_1 t=t^{-1} d_1, d_1^2=1\rangle,
\end{align*}
which has no centre,
it follows that $Z_M\cong\mathbb{Z}$, generated by $d_1^2$.

This manifold can also be given a Seifert bundle structure over a disc with two arc mirrors. The generator $d_1^2$ of $Z_M$ becomes the regular fibre in this second Seifert bundle structure. Note that this then agrees with our findings in Lemma \ref{sbwithbdylemma}. The annuli that become the exceptional fibres are of the form \(\rho\times I\) in \(\kb\times I\), where \(\rho\) is a core curve of a cross-cap in \(\mathbb{K}\).

Finally, suppose that $S$ is an annulus with no marked points and one boundary component a simple closed curve mirror, and that the boundary of $M$ is a Klein bottle. Then
\begin{align*}
\pi_1(M)=& \langle d_1, f_1,g_1, t\mid d_1 t=t^{-1} d_1, g_1^2=t, f_1 g_1=g_1^{-1} f_1, t^r= d_1 f_1\rangle\\
=& \langle f_1,g_1, t\mid f_1 t=t^{-1} f_1, g_1^2=t, f_1 g_1=g_1^{-1} f_1\rangle\\
=& \langle f_1,g_1\mid f_1 g_1=g_1^{-1} f_1\rangle.
\end{align*}
Therefore $Z_M\cong\mathbb{Z}$, generated by $f_1^2$.
Again, we can find an alternative Seifert bundle structure on $M$, this time over a disc with one arc mirror and one marked point. The regular fibre in this new structure corresponds to $f_1^2$. As before, this agrees with Lemma \ref{sbwithbdylemma}.
\end{proof}

\begin{lemma}\label{torusxIlemma}
Suppose that $M$ is Seifert fibred with base space an annulus with no marked points or mirrors, and that $M$ is orientable. Then $M=\ts\times I$ and $Z_M=\pi_1(M)\cong\mathbb{Z}^2$. 
\end{lemma}

\begin{lemma}\label{MbandxS1lemma}
Suppose that one of the following holds:
\begin{itemize}
\item $M$ is Seifert fibred with base space a Mobius band with no marked points, and $M$ is non-orientable;
\item $M$ is a Seifert bundle over an annulus with no marked points and one boundary component a simple closed curve mirror, and the boundary of $M$ is a torus.
\end{itemize}
Then $Z_M\cap\pi_1^+(M)\cong\mathbb{Z}^2$. Moreover, there are two elements of $Z_M\cap\pi_1^+(M)$ that together generate this group, such that each is given by the regular fibre of a Seifert bundle structure on $M$.
\end{lemma}
\begin{proof}
In the first case, 
\begin{align*}
\pi_1(M)=& \langle a_1,d_1,t\mid a_1 t=t a_1, d_1 t=t d_1, t^r=a_1^2 d_1\rangle\\
=& \langle a_1,t\mid a_1 t=t a_1\rangle
\end{align*}
and $Z_M\cap\pi_1^+(M)$ is generated by $t$ and $a_1^2$.

In the second case, 
\begin{align*}
\pi_1(M)=& \langle d_1, f_1,g_1, t\mid d_1 t=t d_1, g_1^2=t, f_1 g_1=g_1 f_1, t^r= d_1 f_1\rangle\\
=& \langle f_1,g_1, t\mid f_1 t=t f_1, g_1^2=t, f_1 g_1=g_1 f_1\rangle\\
=& \langle f_1,g_1\mid f_1 g_1=g_1 f_1\rangle,
\end{align*}
and $Z_M\cap\pi_1^+(M)$ is generated by $f_1$ and $g_1^2=t$.

In fact, each of these manifolds is \(\mb\times\crcle\),  
and there is a homeomorphism between them that takes $g_1$ to $a_1$ and $f_1$ to $t$.
\end{proof}

\subsection{Closed Seifert bundles}

We now need to consider closed Seifert bundles. If $M$ is large, closed and Seifert fibred, Proposition \ref{largenormalprop} shows that $Z_M\subseteq \langle t\rangle$, where $t$ represents a regular fibre. We will return to the list of small Seifert fibred spaces after forming a related list of Seifert bundles that contain fibred solid Klein bottles.

\begin{lemma}\label{sbnobdylemma}
Suppose that $M$ is a closed Seifert bundle with base space $S$. Assume that $S$ has (orientable or non-orientable) genus $x$, $z$ marked points and $w$ simple closed curve mirrors, where $w\geq 1$. Let $t$ denote a regular fibre in $M$.
Assume further that none of the following holds:
\begin{itemize}
\item $S$ is a disc with mirrored boundary and at most one marked point;
\item $S$ is a disc with mirrored boundary and two marked points, and cutting $M$ along the torus of exceptional fibres gives the orientable $\crcle$--bundle over a Mobius band;
\item $S$ is an annulus with two mirrored boundary components and no marked points;
\item $S$ is a Mobius band with mirrored boundary and no marked points.
\end{itemize}
Then either $Z_M\cap\pi_1^+(M)=\langle t\rangle\cong \mathbb{Z}$ (if the fibres can be coherently oriented) or $Z_M\cong 1$ (else).
\end{lemma}
\begin{proof}
Let $T$ be the union of the exceptional fibres in $M$ that have solid Klein bottle neighbourhoods. Then $T$ is a (non-empty) union of properly embedded one-sided tori and Klein bottles. Let $M'$ be the manifold given by cutting $M$ along $T$ and then doubling the resulting manifold (that is, take two copies and glue the two boundaries together by the identity map). Then $M'$ is a double cover of $M$, and the division of $M$ into copies of $\crcle$ induces such a division on $M'$. 

We find that $M'$ is Seifert fibred with a base space $S'$ that is given by doubling $S$ along its boundary components (after forgetting the mirrors).
If $S$ is orientable then $S'$
has $2z$ marked points and genus $2x+w-1$. 
On the other hand, if $S$ is non-orientable then $S'$ has $2z$ marked points and $2x+2(w-1)$ cross-caps.
Moreover, a regular fibre in $M$ lifts to a regular fibre in $M'$.

This process does not affect whether the fibres can be coherently oriented. Thus the fibres of $M$ can be coherently oriented if and only if the fibres of $M'$ can be coherently oriented.

As $M'$ double covers $M$, we can see $\pi_1(M')$ as an index $2$ subgroup of $\pi_1(M)$. More precisely, it is the subgroup of loops in $M$ that cross $T$ an even number of times. Therefore, if $\rho\in Z_M$ then $\rho^2\in\pi_1(M')$ and so $\rho^2\in Z(\pi_1(M'))$.

First suppose that $M'$ is large. By combining Proposition \ref{largenormalprop} and Remark \ref{orientremark} we find that $t$ has infinite order in \(\pi_1(M')\), and either $Z(\pi_1(M'))\cong 1$ or $Z(\pi_1(M'))=\langle t\rangle$, according to whether there is a loop in $S'$ that reverses the fibre. 

If $Z(\pi_1(M'))\cong 1$ then every element of $Z_M$ has finite order. 
It is not possible that $M$ is either an $\sphere$--bundle over $\crcle$ or $\mathbb{P}^2\times\crcle$, as these have fundamental groups that are abelian and contain elements of infinite order. Since $M$ is non-orientable, from Theorem \ref{rp2thm} we conclude that $M$ is irreducible and contains no two-sided $\mathbb{P}^2$.
Theorem \ref{finiteorderthm} therefore shows that $\pi_1(M)$ contains no non-trivial elements of finite order. Thus $Z_M\cong 1$ in this case.

Suppose instead that $Z(\pi_1(M'))=\langle t\rangle$. 
This will be the case exactly when every loop in $S'$ preserves the fibre in $M'$, and so the fibres can be coherently oriented. Translating this back to $M$, we find that every loop in $S$, including each of the (mirrored) boundary components, preserves the fibre in $M$. In particular, from this we know that $\langle t\rangle\leq Z_M$ (again, see Remark \ref{orientremark}).

Let $M''$ be the orientation double cover of $M$ (that is, the cover corresponding to $\pi_1^+(M)$). Then $M''$ is Seifert fibred, and a regular fibre in $M$ lifts to a regular fibre in $M''$. 
Each fibred solid torus in \(M\) lifts to two solid tori in \(M''\), while each fibred solid Klein bottle lifts to one solid torus.
Denote the base space of $M''$ by $S''$. Consider an orientation-reversing loop $\rho$ in $S$. Because $\rho$ preserves the fibre, any corresponding loop in $M$ is orientation-reversing in $M$. This shows that $\rho$ does not lift to a loop in $S''$. We therefore see that $S''$ is orientable. Note also that $S''$ is closed. We additionally know the Euler characteristic of $S''$: if $S$ is orientable then $\chi(S'')=2(-2x+2-w)$, so $S''$ has genus $2x-1+w$, and if $S$ is non-orientable then $\chi(S'')=2(-x+2-w)$, so $S''$ has genus $x-1+w$.

In most cases it follows that $M''$ is large, and so $Z(\pi_1(M''))\leq\langle t\rangle$. 
The possible exceptions occur when $S''$ has genus at most $1$, or equivalently when either $S$ is orientable with $(x,w)\in\{(0,1),(0,2)\}$ or $S$ is non-orientable with $(x,w)=(1,1)$. However, if $S$ is orientable then $M''$ is the same cover as $M'$, and we are currently working under the assumption that $M'$ is large. On the other hand, if $S$ is non-orientable with $x=w=1$ then $S''$ has genus $1$ and $S'$ is formed of two cross-caps. The assumption that $M'$ is large then means that $M'$ contains at least one exceptional fibre. It follows that $M''$ also contains at least one exceptional fibre, and so is large. Hence in each of the cases we are currently considering we find that $M''$ is large. Since $\pi_1(M'')=\pi_1^+(M)$, we conclude that $Z_M\cap\pi_1^+(M)\leq Z(\pi_1(M''))=\langle t\rangle$.

Now suppose that $M'$ is small. We have the following possibilities.
\begin{itemize}
\item $S'$ is a sphere with either zero or two marked points. For this to occur, $z\leq 1$ and $2x+w-1=0$. By assumption $w\geq 1$, so necessarily $w=1$ and $x=0$. Thus $S$ is a disc with mirrored boundary and at most one marked point. We have excluded this case by hypothesis. 

\item $S'$ is a sphere with four marked points. In this case, $z=2$, $w=1$ and $x=0$, so $S$ is a Mobius band with two marked points and mirrored boundary.
Note that, by construction, $M$ cut along the exceptional torus is one component of $M'$ cut along a torus of regular fibres with two exceptional fibres on each side. Although there are different tori in $M'$ that meet this description, the result of cutting along the torus is always the same up to homeomorphism preserving the Seifert fibred structure. Each such piece is homeomorphic to the orientable $\crcle$--bundle over a Mobius band (see \cite{MR0426001} Section 5.4). Our hypotheses exclude this case.

\item $S'$ is a torus with no marked points. Then $z=0$ and $2x+w-1=1$, which implies that $x=0$ and $w=2$. That is, $S$ is an annulus with no marked points, with mirrors on both boundary components. Again, we have excluded this case by hypothesis. 

\item $S'$ is a Klein bottle with no marked points. Thus $z=0$ and $2x+2w-2=2$. Since $x\geq 1$ and $w\geq 1$, we have $x=w=1$, and $S$ is a Mobius band with no marked points and the boundary a mirror. Once again, this case is excluded. \qedhere
\end{itemize}
\end{proof}

\begin{lemma}\label{spheretwopointslemma}
Suppose that $M$ is Seifert fibred over a sphere with at most two marked points. Then either $Z_M\cong 1$ or $Z_M$ is cyclic (finite or infinite according to whether $\pi_1(M)$ is), generated by the regular fibre of some Seifert fibration of $M$.
\end{lemma}
\begin{proof}
This is discussed in \cite{MR0426001} Section 5.4 (which is based on \cite{zbMATH03281474}). There are three possibilities: $M$ is either $\Sphere$, $\sphere\times\crcle$, or a lens space. In each case $Z_M=\pi_1(M)$. 
\end{proof}

\begin{lemma}\label{smalllemma1}
Suppose that $M$ is a small Seifert fibred space, with a base space $S$ such that one of the following holds:
\begin{itemize}
\item $S$ is a sphere with three marked points;
\item $S$ is $\mathbb{P}^2$ with at most one marked point, and $M$ is orientable.
\end{itemize}
Then either $M=\mathbb{P}^3\#\mathbb{P}^3$, in which case $Z_M\cong 1$, or else $Z_M$ is a cyclic group, generated by the regular fibre of some Seifert fibration of $M$, the fibres of which can be coherently oriented.
\end{lemma}
\begin{proof}
Again, this is discussed in \cite{MR0426001} Section 5.4. 

In the case that $S$ is $\mathbb{P}^2$, it may be that $M=\mathbb{P}^3\#\mathbb{P}^3$, to which we may apply Corollary \ref{nonprimecor}. Otherwise, $M$ has a second Seifert fibration, over a sphere with three marked points. 

If $S$ is a sphere with three marked points then the fibres of the fibration can be coherently oriented, so $\langle t\rangle\leq Z_M$. On the other hand, \cite{MR0426001} tells us that $Z_M\leq \langle t\rangle$ (if $\pi_1(M)$ is infinite then \cite{MR0426001} Section 5.3 Lemma 1 applies).
\end{proof}

\begin{lemma}\label{smalllemma2}
Suppose that $M$ is Seifert fibred with base space $S$, and one of the following holds:
\begin{itemize}
\item $S$ is $\mathbb{P}^2$ with at most one marked point, and $M$ is non-orientable;
\item $S$ is a disc with at most one marked point and mirrored boundary.
\end{itemize}
Then $Z_M\cap\pi_1^+(M)\cong\mathbb{Z}$, generated by the regular fibre of some Seifert fibration of $M$ in which the fibres can be coherently oriented.
\end{lemma} 
\begin{proof}
First suppose that $S$ is a disc with at most one marked point and mirrored boundary.
If there are no marked points then
\begin{align*}
\pi_1(M)=& \langle  f_1,g_1,t\mid g_1^2=t, f_1 g_1=g_1 f_1, t^r= f_1\rangle\\
=& \langle  g_1,t\mid g_1^2=t\rangle.
\end{align*}
Thus $Z_M=\pi_1(M)$, generated by $g_1$ (which is orientation-reversing), and $\pi_1^+(M)\cap Z_M\cong\mathbb{Z}$, generated by $t$.

If instead $M$ has one exceptional fibre, we can find an alternative Seifert bundle structure. This can be pictured as follows.
Cutting open \(M\) along the torus of exceptional fibres gives a solid torus. Fix a meridian-longitude pair \((\mu,\lambda)\), and fibre the solid torus by fibres parallel to \(\lambda\).  As discussed in Section \ref{bundlessubsection}, we may assume that the re-gluing map is given by a translation halfway along one of the slopes \((0,1)\), \((1,0)\)  or \((1,1)\) relative to \(\mu\) and \(\lambda\).

If the slope is \((0,1)\) then we have expressed \(M\) as a Seifert bundle over a disc with no marked points and with mirrored boundary.
On the other hand, if the slope is either \((1,0)\) or \((1,1)\) then the disc of the base space glues up to form a cross-cap, so \(M\) is Seifert fibred over \(\mathbb{P}^2\) with no marked points where the core curve of \(\mathbb{P}^2\) preserves the fibre.

Suppose now that $S$ is $\mathbb{P}^2$ with at most one marked point, and $M$ is non-orientable.
Here there are in fact only two different manifolds that arise, as noted in \cite{MR0426001} Section 5.4. The first is $\mathbb{P}^2\times\crcle$, for which $\pi_1(M)\cong\mathbb{Z}\times\mathbb{Z}_2$ and $Z_M\cap\pi_1^+(M)\cong\mathbb{Z}$.
The second is the non-orientable $\sphere$--bundle over $\crcle$, for which $\pi_1(M)\cong\mathbb{Z}$ and $Z_M\cap\pi_1^+(M)\cong\mathbb{Z}$. Each has a Seifert fibration over $\mathbb{P}^2$ with no marked points, such that the regular fibre generates $Z_M\cap\pi_1^+(M)$.
We can see this structure in the presentation for $\pi_1(M)$ as follows.
Since there are no marked points,
\begin{align*}
\pi_1(M)=& \langle a_1,t\mid a_1 t=t a_1, t^r=a_1^2\rangle.
\end{align*}
If $r$ is even this is $\mathbb{Z}\times\mathbb{Z}^2$,
generated by $t$ and $a_1^{-1}t^{\frac{r}{2}}$ (see Lemma \ref{powertwolemma}), and \(Z_M\cap\pi_1^+(M)\) is generated by \(t\).
On the other hand, if $r$ is odd this is infinite-cyclic generated by $a^m t^n$ where $rm+2n=1$ (see Lemma \ref{cyclicgrouplemma}). Note that \(a^m t^n\) is orientation-reversing and \(Z_M\cap\pi_1^+(M)\) is generated by \((a^m t^n)^2=t\).
\end{proof}

\begin{lemma}\label{toruspreservelemma}
Suppose that $M$ is Seifert fibred with base space $S$ that is a torus with no marked points, and each curve on $S$ preserves the fibre. Then either $M$ is the three-torus, in which case $Z_M=\pi_1(M)=\mathbb{Z}^3$, or else $Z_M\cong\mathbb{Z}$, generated by the regular fibre of the Seifert fibration.
\end{lemma}
\begin{proof}
Observe that
\begin{align*}
\pi_1(M)=& \langle a_1,b_1,t\mid a_1 t=t a_1,b_1 t=t b_1, t^r=a_1b_1a_1^{-1}b_1^{-1}\rangle.
\end{align*}
If $r=0$ then $M$ is the three-torus, and $\pi_1(M)=\mathbb{Z}^3$. Note that the symmetry of this space allows us to give it a Seifert fibred structure in different ways.
On the other hand, if $r\neq 0$ then Lemma \ref{grouplemma2} shows that $Z_M\cong\mathbb{Z}$, generated by $t$.
\end{proof}

\begin{lemma}\label{zxzlemma}
Suppose $M$ is a Seifert bundle with base space $S$, and one of the following holds:
\begin{itemize}
\item $S$ is a torus with no marked points, and at least one curve on $S$ reverses the fibre;
\item $S$ is a Klein bottle with no marked points, and every curve on $S$ preserves the fibre;
\item $S$ is an annulus with no marked points and two simple closed curve mirrors, and the core curve of the annulus preserves the fibre;
\item $S$ is a Mobius band with no marked points and mirrored boundary, and the core curve preserves the fibre.
\end{itemize}
Then $Z_M\cong\mathbb{Z}^2$, generated by elements each of which is the regular fibre of a Seifert bundle structure for which the fibres can be coherently oriented.
\end{lemma}
\begin{proof}
Suppose first that $M$ is Seifert fibred over a torus, and at least one loop on $S$ reverses the fibre. By suitable choices of generators,
\begin{align*}
\pi_1(M)=& \langle a_1,b_1,t\mid a_1 t=t^{-1} a_1,b_1 t=t^{-1} b_1, t^r=a_1b_1a_1^{-1}b_1^{-1}\rangle.
\end{align*}
Lemma \ref{grouplemma3} tells us $Z_M\cong\mathbb{Z}^2$, generated by $a_1^2$ together with (if \(r\) is odd) $b_1^2$ or (if $r$ is even) $a_1b_1t^{-r/2}$.

Now suppose that $S$ is a Klein bottle with no marked points, and every curve on $S$ preserves the fibre. Then 
\begin{align*}
\pi_1(M)=& \langle a_1,a_2,t\mid a_1 t=t a_1, a_2t=ta_2, t^r=a_1^2 a_2^2\rangle.
\end{align*}
By Lemma \ref{grouplemma7}, \(Z(G)\cong\mathbb{Z}^2\), generated by \(a_1^2\) and \(t\), or alternatively by \(a_2^2\) and \(t\).
Note that the fibres can be coherently oriented.

By \cite{MR0426001} Section 5.4, these two cases give the same spaces, which are Klein bottle bundles over $\crcle$.

Thirdly, suppose that $S$ is an annulus with two simple closed curve mirrors, and the core curve of the annulus preserves the fibre. Then
\begin{align*}
\pi_1(M)=& \langle f_1,g_1,f_2,g_2, t\mid g_1^2=t, g_2^2=t, f_1 g_1=g_1 f_1, f_2g_2=g_2f_2, t^r=f_1 f_2\rangle\\
=& \langle f_1,g_1,g_2\mid g_1^2=g_2^2, f_1 g_1=g_1 f_1, f_1 g_2=g_2 f_1\rangle.
\end{align*}
Thus $f_1,g_1^2\in Z_M$ and 
\begin{align*}
\pi_1(M)/\langle f_1,g_1^2\rangle=& \langle g_1,g_2\mid g_1^2=g_2^2=1\rangle.
\end{align*}
This has no centre, so $Z_M$ is generated by $f_1$ and $g_1^2=t$. 
Again, the fibres can be coherently oriented, and \(f_1\), \(g_1\) and \(g_2\) all have infinite order.

Fourthly, suppose that \(S\) is a Mobius band with mirrored boundary and no marked points, such that the core curve of the Mobius band preserves the fibre.
Then 
\begin{align*}
\pi_1(M)=& \langle a_1, f_1,g_1,t\mid a_1 t=t a_1,g_1^2=t, f_1 g_1=g_1 f_1, t^r=a_1^2  f_1\rangle\\
=& \langle a_1, g_1,t\mid a_1 t=t a_1,g_1^2=t, a_1^{-2} g_1=g_1 a_1^{-2}\rangle\\
=& \langle a_1, g_1\mid a_1 g_1^2=g_1^{2} a_1, a_1^{2} g_1=g_1 a_1^{2}\rangle,
\end{align*}
which maps onto \(\mathbb{Z}^2\).
Both \(a_1^2\) and \(g_1^2\) are central and
\begin{align*}
\pi_1(M)/\langle a_1^2,g_1^2\rangle
=& \langle a_1, g_1\mid a_1^2=g_1^2=1\rangle,
\end{align*}
which has no centre. Therefore \(Z_M\) is generated by \(a_1^2\) and \(g_1^2=t\). Moreover, the fibres can be coherently oriented.

The third and fourth cases are the same two manifolds as occur in both the first case and the second case. This can be seen as follows.

Suppose again that \(M\) is Seifert fibred over a Klein bottle. Cut \(M\) along the torus of fibres over the core curve of one cross-cap. Call the resulting manifold \(M'\). Then \(M'\) inherits a Seifert fibred structure over a Mobius band where the core curve of the Mobius band preserves the fibre. As in Lemma \ref{MbandxS1lemma}, \(M'\) can also be expressed as a Seifert bundle over an annulus with one mirrored boundary component where the core curve of the annulus preserves the fibre.

Choose a meridian-longitude pair \((\mu,\lambda)\) for the boundary torus such that the new fibres have slope \((0,1)\). 
The gluing map from \(M'\) to \(M\) pairs up points on the torus \(\partial M'\). 
We can use this map to define a second slope on the torus. Take a path \(\sigma\) that connects two points that are identified by the gluing map. Combining \(\sigma\) with its image gives a closed curve. As discussed in Section \ref{bundlessubsection}, we can see the resulting slope as equivalent to one of \((0,1)\), \((1,0)\) or \((1,1)\).

If the slope of the gluing map is \((0,1)\)
then regluing forms a mirror. That is, \(M\) is a Seifert bundle over an annulus with two mirrored boundary components.
On the other hand, if the slope is either \((1,0)\) or \((1,1)\), regluing instead forms a cross-cap in the base space. That is, \(M\) has a Seifert bundle structure over a Mobius band with mirrored boundary.

It is also interesting to note that there is a symmetry in the case of a Mobius band with a mirror corresponding to the two ways of cutting open \(M\) to give \(M'\). This gives three Seifert bundle structures on the same manifold, the regular fibres of any two of which generate \(Z_M\).
\end{proof}

We can also understand the manifolds in Lemma \ref{zxzlemma} more concretely as follows.

\begin{lemma}\label{bundleovertoruslemma}
Suppose \(M\) is a Seifert bundle with base space \(S\).

If \(S\) is an annulus with no marked points and two simple closed curve mirrors, and the core curve of the annulus preserves the fibre, then \(M\) is \(\mathbb{K}\times\crcle\).

If \(S\) is a Mobius band with no marked points and mirrored boundary, and the core curve preserves the fibre, then \(M\) is the mapping torus of the involution of \(\mathbb{K}=\mathbb{M}\cup\mathbb{M}\) that interchanges the two copies of \(\mathbb{M}\).
\end{lemma}
\begin{proof}
First suppose \(S\) is an annulus. 
Let \(M'\) be the manifold given by cutting \(M\) along the exceptional fibres. Then \(M'\) is \((\crcle\times I)\times\crcle\), where the fibres of the Seifert fibration are of the form \(\{q\}\times\crcle\). Let \(T'\) be the annulus \(\{p'\}\times I\times\crcle\) for a fixed point \(p'\).
Then we can view \(M'\) as \(\crcle\times T'\).
In regluing the two boundary tori of \(M'\) to give \(M\), the effect on \(T'\) is to glue up each of the two boundary components by an order \(2\) rotation to give \(\mathbb{K}\). As the same is true for each parallel copy of \(T'\) in \(\crcle\times S'\), we see that \(M=\crcle\times\mathbb{K}\).

Suppose instead that \(S\) is a Mobius band. Again, let \(M'\) be the manifold that results from cutting along the exceptional fibres. Choose a non-trivial properly embedded arc on the Mobius band, and cut \(M'\) along the annulus of fibres over the arc. Call the resulting manifold \(M''\). Then \(M''\) is Seifert fibred over a disc, expressed as \(M''=(I\times I)\times\crcle\) where the fibres are of the form \(\{q\}\times \crcle\). 
Denote by \(S''\) the square \(I\times I\times\{p'\}\) and by \(T''\) the annulus \(I\times\{p''\}\times\crcle\) for fixed points \(p'\) and \(p''\).
To recreate \(M'\) from \(M''\), we need to glue together the two annuli \(I\times\partial I\times\crcle\) such that \(S''\) glues up to form a two-sided Mobius band \(S'\). That is, \(M'\) is the mapping torus of the involution of an annulus \(T'\) given by reflection in the core curve of the annulus, where \(T'\) is the copy of \(T''\) in \(M'\).
In regluing the two boundary tori of \(M'\) to give \(M\), the effect on \(T'\) is to glue up each of the two boundary components by an order \(2\) rotation to give \(\mathbb{K}\). As the same is true for each parallel copy of \(T'\) in the mapping torus \(M'\), we see that \(M\) is the mapping torus of the involution of \(\mathbb{K}\) given by reflection in the curve that separates the two Mobius bands.
\end{proof}

\begin{lemma}\label{smalllemma3}
Suppose that $M$ is a Seifert bundle with base space $S$, and one of the following holds:
\begin{itemize}
\item $S$ is a Klein bottle with no marked points, and the core curves of both cross-caps reverse the fibre; 
\item $M$ is a small Seifert fibred space and $S$ is a sphere with four marked points.
\end{itemize}
Then either \(Z_M\cong 1\) or \(Z_M\cong\mathbb{Z}\), generated by the regular fibre of a Seifert fibration in which the fibres can be coherently oriented. 
\end{lemma}
\begin{proof}
Suppose, first, that $S$ is a Klein bottle with no marked points, and the core curves of both cross-caps of $S$ reverse the fibre. Then
\begin{align*}
\pi_1(M)=& \langle a_1,a_2,t\mid a_1 t=t^{-1} a_1, a_2 t=t^{-1} a_2, t^r=a_1^2 a_2^2\rangle
\end{align*}

If \(r\neq 0\) then, by Lemma \ref{grouplemma4}, \(Z_M\cong 1\).

Assume instead that \(r=0\), so
\begin{align*}
\pi_1(M)=& \langle a_1,a_2,t\mid a_1 t=t^{-1} a_1, a_2 t=t^{-1} a_2, 1=a_1^2 a_2^2\rangle.
\end{align*}
Then \(a_1^2\) is central and 
\begin{align*}
\pi_1(M)/\langle a_1^2\rangle=& \langle a_1,a_2,t\mid a_1 t=t^{-1} a_1, a_2 t=t^{-1} a_2, 1=a_1^2= a_2^2\rangle.
\end{align*}
Lemma \ref{grouplemma5} shows that this has no centre. Thus \(Z_M\) is generated by \(a_1^2\).

By \cite{MR0426001} Section 5.4, this is the same manifold as the small Seifert fibred space with base space a sphere and four marked points.
With this fibration, the fibres can be coherently oriented.
\end{proof}

\begin{lemma}\label{smalllemma4}
Suppose that $M$ is a Seifert bundle with base space $S$, and one of the following holds:
\begin{itemize}
\item $S$ is a disc with two marked points and mirrored boundary, and cutting along the torus of exceptional fibres gives the orientable \(\crcle\)--bundle over a Mobius band;
\item $S$ is a Klein bottle with no marked points, where the core curve of one cross-cap 
preserves the fibre and the other reverses it;
\item $S$ is an annulus with mirrors on both boundary components, and the core curve of the annulus reverses the fibre.
\end{itemize}
Then $Z_M\cong\mathbb{Z}$, generated by the regular fibre of a Seifert bundle structure in which the fibres can be coherently oriented.
\end{lemma}
\begin{proof}
Suppose that $S$ is a Klein bottle with no marked points, and the core curve of one cross-cap preserves the fibre while the other reverses it. Then
\begin{align*}
\pi_1(M)=& \langle a_1,a_2,t\mid a_1 t=ta_1, a_2t=t^{-1} a_2, t^r=a_1^2 a_2^2\rangle.
\end{align*}
Now, $a_2^2$ is central and 
\begin{align*}
\pi_1(M)/\langle a_2^2\rangle=& \langle a_1,a_2,t\mid a_1 t=ta_1, a_2t=t^{-1} a_2, t^r=a_1^2, a_2^2=1\rangle,
\end{align*}
which has no centre, as shown by Lemma \ref{grouplemma6}, so \(Z_M\) is cyclic. In addition,
\begin{align*}
\pi_1(M)/\langle\!\langle t, a_1a_2\rangle\!\rangle=& \langle a_2\mid \ \rangle,
\end{align*}
where \(\langle\!\langle t, a_1a_2\rangle\!\rangle\) denotes the normal subgroup generated by \(t\) and \(a_1a_2\), so \(a_2^2\) has infinite order.

Cut \(M\) open along the torus of fibres over the core curve of a cross-cap in \(S\) that preserves the fibre. Call the new manifold \(M'\). Then \(M'\) inherits a Seifert fibred structure over a Mobius band where the core curve of the Mobius band reverses the fibre. That is, \(M'\) is the orientable \(\crcle\)--bundle over a Mobius band.
As in Lemma \ref{isZhasbdylemma}, \(M'\) can be refibred over a disc with two marked points.
The fibres of the new fibration define a slope on the boundary torus \(\partial M'\). 
As in the proof of Lemma \ref{zxzlemma}, we can ask whether this slope coincides with that of the gluing map.

If the two slopes do not agree then the boundary component of the base surface of the Seifert fibre structure closes up to form a cross-cap. That is, \(M\) is Seifert fibred over \(\mathbb{P}^2\) with two marked points, where the core curve of the cross-cap preserves the fibre. 
Here the fibres can be coherently oriented. In this case Proposition \ref{largenormalprop} applies, showing that \(Z_M\cong\mathbb{Z}\), generated by the regular fibre.

Suppose instead that the slope of the fibres aligns with the gluing map. 
Then regluing gives a mirror. That is, \(M\) is expressed as a Seifert bundle over a disc with mirrored boundary and two marked points; we are now in the first case in the statement of Lemma \ref{smalllemma4}. The fibres can be coherently oriented, and
\begin{align*}
\pi_1(M)= \langle c_1,c_2,f_1&,g_1, t'\mid c_1 t'=t' c_1,c_2 t'=t' c_2, c_1^2=t',c_2^2=t',  g_1^2=t', \\
&f_1 g_1=g_1 f_1, (t')^n=c_1 c_2 f_1\rangle\\
= \langle c_1,c_2,g_1&, t'\mid c_1^2=t',c_2^2=t',  g_1^2=t', g_1c_1c_2= c_1c_2g_1\rangle\\
= \langle c_1,c_2,g_1&\mid c_1^2=c_2^2= g_1^2, g_1c_1c_2= c_1c_2g_1\rangle.
\end{align*}
There is an isomorphism \(\theta\) to the group presentation above, with \(r=0\), given by
\(\theta(c_1)=a_2\), \(\theta(c_2)=a_2t\) and \(\theta(g_1)=a_1^{-1}\). Thus with this presentation \(Z_M\) is generated by \(\theta^{-1}(a_2^2)=c_1^2=t'\).

Now suppose instead that \(S\) is an annulus with mirrors on both boundary components, where the core curve of the annulus reverses the fibre. Then
\begin{align*}
\pi_1(M)=& \langle f_1,g_1,f_2,g_2, t\mid g_1^2=t, g_2^2=t, f_1 g_1=g_1^{-1} f_1, f_2g_2=g_2^{-1}f_2, t^r=f_1 f_2\rangle\\
=& \langle f_1,g_1,g_2\mid g_1^2=g_2^2, f_1 g_1=g_1^{-1} f_1, f_1 g_2=g_2^{-1} f_1\rangle.
\end{align*}
By Lemma \ref{equalsquareslemma}, the centre of this group is generated by \(f_1^2\).

This manifold is also homeomorphic to the second case we have just considered. That is, it can be refibred as a Seifert fibred space over a Klein bottle where the core curve of one cross-cap preserves the fibre and another reverses it, and also as a Seifert bundle over a disc with two marked points and mirrored boundary. This second refibring can be pictured as follows. 

 Let $M'$ be a Seifert bundle with base space $S'$, where $S'$ is an annulus with one boundary component a simple closed curve mirror, such that $\partial M'$ is a Klein bottle (that is, the core curve of the annulus reverses the fibre). Then doubling $M'$ along its boundary gives $M$. In Lemma \ref{isZhasbdylemma} we noted that $M'$ can also be expressed as a Seifert bundle over a disc with one arc mirror and one marked point. 
Consider the boundary patterns created by these fibrings. In each case, there are two curves that are the boundary of the annulus of exceptional fibres. These are the two distinct non-separating, orientation-reversing simple closed curves on a Klein bottle. These cut the Klein bottle into an annulus, and the regular fibre fills this with circles (being the unique separating curve).
Therefore the two fibrings match up suitably under the re-gluing identification.
\end{proof}

\begin{lemma}\label{mbandclosedlemma}
Suppose that $M$ is a Seifert bundle with base space a Mobius band with mirrored boundary, such that the core curve of the Mobius band reverses the fibre.
Then $Z_M\cap\pi_1^+(M)\cong\mathbb{Z}$, generated by the regular fibre of a Seifert bundle structure on $M$.
\end{lemma}
\begin{proof}

The boundary curve of the Mobius band must preserve the fibre, so
\begin{align*}
\pi_1(M)=& \langle a_1, f_1,g_1,t\mid a_1 t=t^{-1} a_1,g_1^2=t, f_1 g_1=g_1 f_1, t^r=a_1^2  f_1\rangle\\
=& \langle a_1, g_1,t\mid a_1 t=t^{-1} a_1,g_1^2=t, a_1^{-2} g_1=g_1 a_1^{-2}\rangle\\
=& \langle a_1, g_1\mid a_1 g_1^2=g_1^{-2} a_1, a_1^{2} g_1=g_1 a_1^{2}\rangle.
\end{align*}
Now, $a_1^2$ is central and
\begin{align*}
\pi_1(M)/\langle a_1^2\rangle
=& \langle a_1, g_1\mid a_1 g_1^2=g_1^{-2} a_1, a_1^{2}=1\rangle,
\end{align*}
which has no centre by Lemma \ref{squarereverselemma}, so $Z_M$ is generated by $a_1^2$.

This space is also Seifert fibred over \(\mathbb{P}^2\) with two marked points, as can be pictured as follows. 
Cutting along the torus of exceptional fibres gives the orientable manifold Seifert fibred over a Mobius band with no marked points. As noted in Lemma \ref{isZhasbdylemma}, this can also be given a Seifert fibred structure over a disc with two marked points. Regluing to give \(M\) causes the boundary of the disc to close up into a cross-cap.

This structure can be observed in the fundamental group by noting that
\begin{align*}
\langle a_1,c_1,c_2,t&\mid a_1t=ta_1, c_1t=tc_1,c_2t=tc_2,c_1^2=t,c_2^2=t,t=a_1^2c_1c_2\rangle\\
&=\langle a_1,c_1,t\mid a_1t=ta_1,c_1^2=t,(c_1^{-1}a_1^{-2}t)^2=t\rangle\\
&=\langle a_1,c_1\mid a_1c_1^2=c_1^2a_1,a_1^{-2}c_1=c_1 a_1^2\rangle.
\end{align*}
With this presentation, the centre of the group is generated by $c_1^2=t$, corresponding to a regular fibre. This agrees with Proposition \ref{largenormalprop}, which applies in this case.
\end{proof}

\section{The main proof}\label{proofsection}

The following theorems are versions of the Seifert Fiber Space Theorem. Discussion of this, and references for the proofs, can be found in \cite{MR3178943}.

\begin{theorem}[\cite{MR1303675} Theorem A]\label{orientablesfsthm}
Let $M'$ be a compact, orientable, irreducible $3$--manifold such that $\pi_1(M')$ is infinite. Then $M'$ is Seifert fibred if and only if $\pi_1(M')$ contains a non-trivial cyclic normal subgroup.
\end{theorem}

\begin{theorem}[\cite{MR1303675} Theorem 1]\label{nonorientablesfsthm}
Let $M'$ be a compact, non-orientable, irreducible $3$--manifold. If $\pi_1(M')$ contains a non-trivial cyclic normal subgroup then either $M'$ is a Seifert bundle or $M'$ has a \(\mathbb{P}^2\) boundary component.
\end{theorem}

\begin{proposition}\label{primesbprop}
Let $M$ be a compact $3$--manifold with no boundary component that is either an $\sphere$ or a $\pp^2$. If $\pi_1(M)$ has a non-trivial centre then $M$ is prime and is a Seifert bundle.
\end{proposition}
\begin{proof}
Let $\rho$ be a non-trivial element of the centre $Z_M$ of $\pi_1(M)$. 
Then $\rho$ generates a (non-trivial) cyclic normal subgroup of $\pi_1(M)$.
Corollary \ref{nonprimecor} tells us that $M$ is prime.

To see that $M$ is a Seifert bundle, we must consider different cases.
Firstly, if $M$ is reducible then $M$ is either the product space $\sphere\times\crcle$ or the twisted bundle $\sphere\rtimes\crcle$ (see, for example, \cite{MR0415619} Lemma 3.13), both of which are Seifert fibred. 
Assume further, therefore, that $M$ is irreducible.
If $M$ is non-orientable then Theorem \ref{nonorientablesfsthm} applies, showing that $M$ is a Seifert bundle.
If $M$ is orientable and $\pi_1(M)$ is infinite then Theorem \ref{orientablesfsthm} gives that $M$ is Seifert fibred.
Finally, if $M$ is orientable and $\pi_1(M)$ is finite then we again find that $M$ is Seifert fibred, by Theorem \ref{finitepi1thm}.
\end{proof}

\maintheorem
\begin{proof}
By Proposition \ref{sequenceprop}, the exact sequence holds with $K=\ker(\Phi_M)$.
We need to consider the case that $K\ncong 1$.
It then follows, by Lemma \ref{centrelemma}, that $Z_M\ncong 1$, where \(Z_M\) denotes the centre \(Z(\pi_1(M))\) of \(\pi_1(M)\). 
By Proposition \ref{simplebdyprop}, no boundary component of $M$ is either $\sphere$ or $\mathbb{P}^2$. Accordingly, Proposition \ref{primesbprop} tells us that $M$ is a prime Seifert bundle.
Consider the presentation for $\pi_1(M)$ given by Proposition \ref{sbpi1prop}.
Recall that if the fibres of the Seifert bundle structure can be coherently oriented then there is an induced $\crcle$--action on $M$ demonstrating that the regular fibre is an element of $K$.

If $M$ has boundary (that is, if $y\geq 1$) then, apart from the five exceptions in Lemma \ref{sbwithbdylemma}, $Z_M$ is either trivial (a possibility we have already excluded) or infinite cyclic with generator $t$. Moreover, $Z_M\ncong 1$ exactly when the fibres can be coherently oriented, demonstrating that $t\in\ker(\Phi_M)$. Hence we find that $K=Z_M$.
Similarly, Lemma \ref{isZhasbdylemma} shows that in some (subcases) of these five exceptional cases $Z_M\cap\pi_1^+(M)\cong\mathbb{Z}$, generated by the regular fibre of some Seifert bundle structure in which the fibres can be coherently oriented.
There are three remaining possibilities. If $M$ is orientable and Seifert fibred over an annulus with no marked points (see Lemma \ref{torusxIlemma}) then $M$ is \(\ts^2\times I\).
On the other hand, as mentioned in Lemma \ref{MbandxS1lemma}, $M$ is $\mathbb{M}\times\crcle$ if $M$ is Seifert fibred with no exceptional fibres over a base space $S$, where either \(S\) is a Mobius band and $M$ is non-orientable, or else $S$ is an annulus and $\partial M$ is a torus.
This gives us the first two exceptions listed in Theorem \ref{maintheorem}.

We now assume that $M$ is closed (that is, that $y=0$, and therefore also $v=0$). If the given Seifert bundle structure is not a Seifert fibration (that is, if $w\geq 1$) then applying Lemma \ref{sbnobdylemma} leaves only four cases to consider.
The first of these is covered by Lemma \ref{smalllemma2}, the second by Lemma \ref{smalllemma4},
the third by Lemmas \ref{zxzlemma} and \ref{smalllemma4},
and the fourth by Lemmas \ref{zxzlemma} and \ref{mbandclosedlemma}. Lemma \ref{zxzlemma} together with Lemma \ref{bundleovertoruslemma} adds to the list of exceptional cases.

Finally, suppose that the given structure on $M$ is a Seifert fibration (that is, that $w=0$). If $M$ is large then Proposition \ref{largenormalprop} says that $Z_M\leq\langle t\rangle$ and $t$ has infinite order. Combining this with Remark \ref{orientremark}, we see that $Z_M=\langle t\rangle$ and the fibres can be coherently oriented, since $Z_M\ncong 1$.

The various cases when $M$ is small are addressed in Lemmas \ref{spheretwopointslemma}, \ref{smalllemma1}, \ref{smalllemma2}, \ref{toruspreservelemma}, \ref{zxzlemma}, \ref{smalllemma3} and \ref{smalllemma4}.
In each case, Theorem \ref{maintheorem} is satisfied. The final exceptional case in the statement is given by Lemma \ref{toruspreservelemma}, with Lemma \ref{zxzlemma} repeating our earlier findings.
\end{proof}

It is interesting to note the geometry of the list of exceptional cases in Theorem \ref{maintheorem}.
The closed manifolds have Euclidean geometry.
The interior of the manifold \(\ts^2\times I\) does not admit a finite volume geometric structure. On the other hand, as a manifold with boundary it has a finite volume Euclidean structure with geodesic boundary. Quotienting this by an isometry yields \(\mb\times\crcle\).


\bibliography{birmanreferences}
\bibliographystyle{hplain}


\bigskip
\noindent
University of Hull

\noindent
Hull, HU6 7RX, UK

\smallskip
\noindent
\textit{jessica.banks[at]lmh.oxon.org}

\noindent
\textit{j.banks[at]hull.ac.uk}

\end{document}